\documentclass[a4paper]{article}
\usepackage[english]{babel}
\usepackage[utf8]{inputenc}
\usepackage[T1]{fontenc}

\usepackage[a4paper,top=3cm,bottom=2cm,left=2.5cm,right=2.5cm,marginparwidth=1.75cm]{geometry}

\usepackage{
	cite,
	amsmath,
	amsthm,
	amssymb,
	bbm,
	mathtools,
	physics,
	braket, 
	tikz,
	enumitem,
	esint,
	changepage
}
\usepackage[colorlinks=true, allcolors=blue]{hyperref}

\newcommand{\Contraction}{\mathrel{\scalebox{1.8}{$\llcorner$}}}

\numberwithin{equation}{section}

\newtheorem{teorema}{Theorem}[section]
\newtheorem{lema}[teorema]{Lemma}

\newtheorem{prop}[teorema]{Proposition}

\newtheorem{defi}[teorema]{Definition}
\newtheorem{remark}[teorema]{Remark}

\DeclareMathOperator{\sign}{sign}

\newcommand{\R}{\mathbb{R}}
\newcommand{\N}{\mathbb{N}}

\DeclareMathOperator{\AIB}{AIB}

\DeclareMathOperator{\cof}{cof}

\DeclareMathOperator{\Id}{Id}
\DeclareMathOperator{\Div}{Div}
\DeclareMathOperator{\adj}{adj}
\DeclareMathOperator{\divergencia}{div}

\DeclareMathOperator{\T}{T}

\newcommand{\Rd}{\mathbb{R}^{d}}
\DeclareMathOperator{\imT}{im_T}
\DeclareMathOperator{\imG}{im_G}
\renewcommand{\O}{\Omega}
\newcommand{\p}{\partial}
\newcommand{\bApO}{\overline{\mathcal{A}}_p(\Omega)}
\newcommand{\ApO}{\mathcal{A}_p(\Omega)}

\usepackage[affil-it]{authblk}

\title{\textbf{Invertibility of Sobolev maps through approximate invertibility at the boundary and tangential polyconvexity}}
\author[1]{\textsc{Carlos Mora-Corral}\footnote{carlos.mora@uam.es}}
\author[2]{\textsc{David Mur-Callizo}\footnote{david.mur@estudiante.uam.es}}
\affil[1]{\normalsize \textup{Departamento de Matemáticas, Universidad Autónoma de Madrid, 28049 Madrid, Spain; and\\
Instituto de Ciencias Matemáticas, CSIC-UAM-UC3M-UCM, 28049 Madrid, Spain.}}
\affil[2]{\normalsize \textup{Departamento de Matemáticas, Universidad Autónoma de Madrid, 28049 Madrid, Spain.}}
\date{}

\begin{document}
	\maketitle
	
	\textbf{Submitted:} August, 2024; \textbf{Accepted:} February, 2025.
	\begin{adjustwidth}{25pt}{25pt}
	\section*{Abstract}
		We work in a class of Sobolev $W^{1,p}$ maps, with $p > d-1$, from a bounded open set $\Omega \subset \Rd$ to $\Rd$ that do not exhibit cavitation and whose trace on $\partial \Omega$ is also $W^{1,p}$.
		Under the assumptions that the Jacobian is positive and the deformation can be approximated on the boundary by injective maps, we show that the deformation is injective.
		We prove the existence of minimizers in this class for functionals accounting for a nonlinear elastic energy and a boundary energy. 
		The energy density in $\O$ is assumed to be polyconvex, while the energy density in $\p \O$ is assumed to be 
		tangentially polyconvex, a new type of polyconvexity on $\partial \Omega$.

	\section*{Resumé}
		Nous travaillons sur une classe de fonctions Sobolev $W^{1,p}$, avec $p > d - 1$, d'un ensemble ouvert et borné $\Omega \subset \Rd$ vers $\Rd$ qui ne présentent pas de cavitation et dont la trace sur $\partial \Omega$ est également $W^{1,p}$. Sous les hypothèses que le Jacobien est positif et que la déformation peut être approximée sur la frontière par des fonctions injectives, nous montrons que la déformation est injective. Nous prouvons l'existence de minimiseurs dans cette classe pour des fonctionnelles considérant d'une énergie élastique non-linéaire et d'une énergie de frontière. La densité d'énergie dans $\O$ est supposée polyconvexe, tandis que la densité d'énergie dans $\p \O$ est supposée tangentiellement polyconvexe, un nouveau type de polyconvexité sur $\p \O$.
	\end{adjustwidth}
	
	\textbf{Keywords:} Approximate invertibility, global invertibility, Sobolev maps, nonlinear elasticity, polyconvexity.
	
	\textbf{MSC 2020:} 49J40, 49J45, 74B20, 74G25, 74G65.
	
	\section{Introduction}
	
	A classic problem in topology is to prove invertibility of a map from local invertibility and  invertibility at the boundary.
	This question has also a long history in nonlinear elasticity theory, where a deformation map $u : \O \to \Rd$ is assumed to be Sobolev $W^{1,p}$ from a bounded open set $\O \subset \Rd$ representing the reference configuration.
	Local invertibility and preservation of orientation are modelled through the constraint $\det D u > 0$ a.e., where $Du$ is the deformation gradient.
	Invertibility on the boundary is typically imposed with an adequate Dirichlet boundary condition.
	The goal is then to obtain invertibility for $u$, since this property is physically required in order to prevent interpenetration of matter.   
	
	The pioneering work of Ball \cite{Ball81} showed that when $p\geq d$, Sobolev maps are invertible a.e.\ or even homeomorphisms when they coincide in $\p \O$ with an invertible map.
	Here, \emph{invertibility a.e.}\ means that the restriction of $u$ to a set of full measure is invertible.
	Further developments of invertibility in the context of nonlinear elasticity are \cite{CiNe87,Sverak,Muller-Qi-Yan,Muller-Spector,HeMo10,HeMo12}.
	
	In this article we focus on the approaches of Henao, Mora-Corral and Oliva \cite{Henao-Mora-Oliva} and Kr\"omer \cite{Kromer}, which we explain next.
	In \cite{Henao-Mora-Oliva} it was defined the class $\bApO$ of Sobolev maps $W^{1,p} (\O, \Rd)$, with $p > d-1$, such that its trace belongs to  $W^{1,p} (\p\O, \Rd)$ and satisfy the \emph{divergence identities} (see \cite{Muller,Sverak,Muller90CRAS,Muller-Qi-Yan})
	\begin{equation}\label{eq:divergence}
		\Div \left[ \adj D u (x) \, g(u(x)) \right] = \divergencia g(u(x)) \det D u(x)
	\end{equation}
	up to the boundary, meaning that for all $\phi \in C^{\infty} (\bar{\O})$ and $g \in C^1 (\Rd,\Rd) \cap W^{1,\infty} (\Rd,\Rd)$ we have
	\begin{equation}\label{eq:introE=Fphig}
		\begin{split}
			& \int_{\p \O} \phi (x) \left( \adj D u(x) \, g(u(x)) \right) \cdot n (x) \, \dd \mathcal{H}^{d-1} (x)  - \int_{\O} \left[ \adj D u(x) \, g(u(x)) \right] \cdot D \phi (x) \, \dd x \\
			& \qquad = \int_{\O} \det D u(x) \, \phi (x) \divergencia g(u(x)) \, \dd x ,
		\end{split}
	\end{equation}
	where $n$ is the unit exterior normal of $\p \O$.
	This class is the version \emph{up to the boundary} of class $\ApO$ defined in \cite{Barchiesi} for which equality \eqref{eq:introE=Fphig} is requested to hold only for $\phi \in C^{\infty}_c (\O)$ (and, hence, the integral on $\p \O$ vanishes).
	It was shown there that maps in $\ApO$ enjoy extra regularity than a typical $W^{1,p}$ function, and, earlier in \cite{Henao-Mora, HeMo12}, that they do not present cavitation.
	The main result in \cite{Henao-Mora-Oliva} is that deformations in $\bApO$ that coincide with an invertible map on $\p \O$ are themselves invertible a.e.
	
	In \cite{Kromer} it is performed a topological study of maps that are \emph{approximately invertible on the boundary}, meaning that their restriction to $\p \O$ can be uniformly approximated by continuous invertible maps.
	The class of such maps is denoted by $\AIB$.
	Then, it was shown that those deformations that, in addition, are Sobolev $W^{1,p}$ with $p \geq d$ and preserve the orientation are invertible a.e.
	An advantage of his approach is that the invertibility condition is only required on the boundary, which avoids the delicate issue of homeomorphic extension and allows for boundary conditions different from Dirichlet.
	Moreover, the notion of approximate invertibility on the boundary permits self-contact at the boundary, while forbiding self-interpenetration.
	
	In this article we extend, in the class $\bApO$ with $p > d-1$, the result of \cite{Kromer} on invertibility a.e.\ from approximate invertibility on the boundary.
	At the same time, it also generalizes the result of \cite{Henao-Mora-Oliva} inasmuch the boundary data is not required to be an invertible a.e.\ map on the whole $\O$ but only approximately invertible on the boundary.
	
	After establishing the invertibility result, in view of its applications in nonlinear elasticity, we study energies of the form
	\[
	I(u) = \int_{\Omega} W(x,u(x),Du(x))\dd x + \int_{\partial\Omega}V(x,u(x),D^{\tau}u(x),n(x))\dd\mathcal{H}^{d-1}(x)
	\]
	in the class $\bApO \cap \AIB$.
	The integral in $\O$ is standard in nonlinear elasticity and accounts for the elastic energy plus volume forces.
	The usual assumption is that $W$ is \emph{polyconvex}, that is, convex in the minors of the derivative.
	The integral in $\p \O$ has not received as much attention as the volume term.
	It accounts for the applied surface forces and, in general, the surface interaction potentials; see, e.g., \cite[Section 5.1]{Ciarlet88} or \cite{PodioVergara}, and, in the context of binary alloys,  \cite{OwenSternberg92}.
	The function $D^{\tau} u$ is the tangential derivative of $u_{| \p \O}$, and $n(x)$ is the normal to $\O$.
	
	A necessary and sufficient condition for the lower semicontinuity of the integral in $\p \O$
	is the tangential quasiconvexity of $V$; see \cite{DaFoMaTr99}.
	In this article we introduce a natural sufficient condition: the \emph{tangential polyconvexity}, which roughly consists of being convex in the minors of the tangential derivative of $u_{|\p \O}$. 
	With this, we prove the existence of minimizers of $I$ in the class $\bApO \cap \AIB$ under several boundary conditions, not necessarily Dirichlet.
	Moreover, we compare our notion of tangential polyconvexity with the related one of \emph{interface polyconvexity} from \cite{Silhavy}. Both concepts refer to maps that are convex in certain minors of the derivative of $u_{|\p \O}$. The interface polyconvexity, originally defined as the supremum of a family of null Lagrangians, requires, in an \textit{a posteriori} characterization, the positive $1$-homogeneity. This is not needed in the definition of tangential polyconvexity, because it is based on the convexity on the minors of the tangential differential of $u_{|\p \O}$, once a basis of the tangent space is chosen.
	
	The outline of this article is as follows.
	In Section \ref{se:notation} we explain the general notation and preliminary concepts and results.
	Section \ref{se:class} recalls the definition and main properties of class $\bApO$.
	In Section \ref{se:injectivity} we prove the injectivity a.e.\ of deformations in $\bApO \cap \AIB$.
	In Section \ref{section counterexample} we show, by means of a counterexample, that the injectivity a.e.\ does not hold in general if the class $\bApO$ is replaced by the class $\ApO$.
	Section \ref{seccion algebra multi} shows the weak continuity in $W^{1,p} (\p \O, \Rd)$ of the minors of the tangential derivative.
	In Section \ref{se:tangential} we define tangential polyconvexity: it implies tangential quasiconvexity and is equivalent to polyconvexity of an extension.
	In Section \ref{sec: Interface polyconvexity}, through the language of multilinear algebra, we give insight into the interface polyconvexity and compare it to the tangential polyconvexity.
	Section \ref{se:examples} shows that typical examples of surfaces potentials used in the literature are tangentially polyconvex.
	Finally, in Section \ref{se:existence} we prove the existence of minimizers of $I$ in $\bApO \cap \AIB$ under the assumptions of polyconvexity of the energy density in $\O$ and tangential polyconvexity of the energy density on $\p \O$.
	
	\section{Notation and preliminaries}\label{se:notation}
	
	We first specify the general notation used in the article.
	
	We will use $\Omega$ to refer to a bounded open subset of $\Rd$, which most of the times will be assumed to have a Lipschitz boundary and that $\Rd\setminus\partial\Omega$ has exactly two connected components.
	Here $d\in\N$ is the dimension of the space, which will be assumed to be $d \geq 2$; otherwise, $\Rd\setminus\partial\Omega$ will have three connected components for an interval $\O$.
	The issue of invertibility for $d=1$ is essentially trivial, since for Sobolev functions is reduced to having positive derivative.
	The set $\Omega$ represents an elastic material in its reference configuration, and $u:\overline{\Omega}\rightarrow\Rd$ is the deformation of the body.
	
	The notation for Sobolev $W^{1,p}$ and Lebesgue $L^p$ spaces is standard.
	Most of the times, the exponent $p$ will satisfy $p>d-1$.
	For $u\in W^{1,p}(\Omega;\Rd)$, we denote by $u_{|\partial\Omega}$ the trace of $u$ on $\partial\Omega$.
	It is known that $u_{|\partial\Omega}\in L^{p}(\partial\Omega;\Rd)$, and we will write $u\in W^{1,p}(\partial\Omega;\Rd)$ whenever its trace $u_{|\partial\Omega}$ belongs to $W^{1,p}(\partial\Omega;\Rd)$.
	
	We will abbreviate \emph{almost everywhere} as \emph{a.e.}, which refers to the $d$-dimensional Lebesgue measure $\mathcal{L}^{d}$, unless otherwise specified.
	We say that two subsets of $\Rd$ are equal a.e.\ if its symmetric difference has $\mathcal{L}^{d}$-measure zero.
	We will also use $\mathcal{H}^{d-1}$ to refer to the $(d-1)$-dimensional Hausdorff measure.
	The set $S^{d-1} \subset \Rd$ is the $d$-dimensional unit sphere.
	
	The set $\R^{d\times d}$ is the set of square matrices of order $d$, and its subset $\R_{+}^{d\times d}$ consists of those with positive determinant.
	The adjoint $\adj F$ of an $F \in \R^{d\times d}$ is the square matrix that satisfies $F \adj F = (\det F) I$, where $I \in \R^{d\times d}$ is the identity matrix.
	The cofactor $\cof F$ is the transpose of $\adj F$.
	We will use $\mathcal{L}(U;V)$ to denote the set of linear maps between two (finite-dimensional) vector spaces $U$ and $V$.

	A key concept studied in \cite{Kromer} is the \emph{approximate invertibility on the boundary}.
	
	\begin{defi}\label{definition AIB}
		Let $\Omega\subset\Rd$ be open and bounded and let $u\in C(\partial\Omega;\Rd)$. We say that $u$ is approximate invertible on the boundary if there exists a sequence of injective maps $\left\{\varphi_{k}\right\}_{k\in\N}\subset C\left(\partial\Omega;\Rd\right)$ with $\varphi_{k}\rightarrow u$ uniformly on $\partial\Omega$. The class of all such maps is denoted by $\AIB(\Omega)$, or by $\AIB$ if $\Omega$ is clear from the context.
		We say that $u : \bar{\O} \to \Rd$ is approximate invertible on the boundary if so is $u_{|\p \O}$.
	\end{defi}
	
	Condition $\AIB$ models possible self-contact at the boundary without interpenetration in the interior, hence it is a realistic class to pose problems in nonlinear elasticity.
	Note that if $u:\partial\Omega\rightarrow\Rd$ is continuous and injective then $u\in\AIB$; and if $u\in\AIB$, then $u$ is continuous on $\partial\Omega$.
	The following lemma is a variant of \cite[Lemma 2.3]{Kromer} and its proof is elementary relying on the fact that $W^{1,p}(\partial\Omega;\Rd)$ is compactly embedded in $C(\partial\Omega;\Rd)$ for $p>d-1$.
	
	\begin{lema}\label{compactness of AIB}
		Let $\O$ be a Lipschitz domain.
		Let $p>d-1$, let $u\in W^{1,p}\left(\partial\Omega;\Rd\right)$ and let $\left\{u_{k}\right\}_{k\in\N}\subset W^{1,p}(\partial\Omega;\Rd) \cap \AIB$ be such that $u_{k}\rightharpoonup u$ in $W^{1,p}(\partial\Omega;\Rd)$.
		Then $u\in\AIB$.
	\end{lema}
	
	We say that a function $u:\Omega\rightarrow\Rd$ is \emph{injective a.e.}\ if there exists some $\widetilde{\Omega}\subset\Omega$ such that $\mathcal{L}^{d}(\Omega\setminus\widetilde{\Omega})=0$ and $u$ is injective in $\widetilde{\Omega}$.
	
	The following proposition is a version of Federer's area formula \cite{Federer}; this specific formulation can be found in \cite[Proposition 2.6]{Muller-Spector}.
	
	\begin{prop}\label{proposicion de Omega0}
		Let $u\in W^{1,1}(\Omega;\Rd)$.
		Then there exists a measurable set $\Omega_{0}\subset\Omega$, with $\mathcal{L}^{d}\left(\Omega\setminus\Omega_{0}\right)=0$, such that the following property holds.
		For any measurable $A\subset\Omega$, define $\mathcal{N}_{u,A}:\Rd\rightarrow\N\cup\{\infty\}$ as follows: $\mathcal{N}_{u,A}(y)$ equals the number of $x\in\Omega_{0}\cap A$ such that $u(x)=y$.
		Then $\mathcal{N}_{u,A}$ is measurable and for any measurable $\varphi:\Rd\rightarrow\R$,
		\begin{equation*}
			\int_{A}\varphi(u(x))\abs{\det Du(x)}\dd x=\int_{\Rd}\varphi(y)\mathcal{N}_{u,A}(y)\dd y,
		\end{equation*}
		whenever either integral exists.
	\end{prop}
	
	Here, $Du$ stands for the distributional derivative of the Sobolev function $u$. We will mainly use $\mathcal{N}_{u,\Omega}$, which will be denoted by $\mathcal{N}_{u}$.
	
	We define the concept of the geometric image of a set $\Omega\subset\Rd$ under a function $u$ (see \cite{Muller-Spector}).
	
	\begin{defi}Let $u\in W^{1,p}\left(\Omega;\Rd\right)$ and let $\Omega_{0}$ be the set of Proposition \ref{proposicion de Omega0}. We define the geometric image of $\Omega$ under $u$ as $\imG(u;\Omega)\coloneqq u\left(\Omega_{0}\right)$.\end{defi}
	
	Note that the set $\Omega_{0}$ described in Proposition \ref{proposicion de Omega0} is not uniquely defined. In particular, if $\Omega_{1}$ is another set with the same properties, then for any measurable $A\subset\Omega$, we have that $u(A\cap\Omega_{0})=u(A\cap\Omega_{1})$ a.e.\ and the two definitions of $\mathcal{N}_{u,A}$ that come from these sets coincide a.e. For example, in \cite{Barchiesi}, $\Omega_{0}$ is chosen as the set of approximate differentiability points of $u$. 
	
	A fundamental tool in this article, as well as in the context of nonlinear elasticity, is Brouwer's degree; see, e.g., \cite[Chapter 1]{Deimling}.
	The degree on $\Omega$ of (the continuous representative of) a map in $W^{1,p}(\partial\Omega;\Rd)$ is defined as the degree of any continuous extension to $\overline{\Omega}$.
	Another important concept is the \emph{topological image} (see \cite{Sverak,Muller-Spector}).
	
	\begin{defi}\label{topological image} Let $\Omega\subset\Rd$ be a bounded domain and let $u\in C\left(\partial\Omega;\Rd\right)$. We define the topological image of $\Omega$ with respect to $u$ as \begin{equation*}\imT\left(u;\Omega\right)\coloneqq\left\{y\in\Rd\setminus u\left(\partial\Omega\right):\deg(u;\Omega;y)\neq0\right\}.\end{equation*}\end{defi}
	
	Note that $\deg\left(u;\Omega;y\right)=0$ for all $y$ in the unbounded component of $\Rd\setminus u\left(\partial\Omega\right)$.
	Therefore $\imT(u;\Omega)$ is a bounded set, and is also open because of the continuity of the degree.
	
	In this article we will assume that $\Rd\setminus\partial\Omega$ has exactly two connected components, which excludes the case $d=1$.
	By the Jordan separation theorem, $\Rd\setminus\partial\Omega$ has exactly two connected components if $\partial \Omega$ is homeomorphic to the sphere, but the converse is not true, as shown by the classic example of the Warsaw circle.
	In addition, $\Omega$ is assumed to have a Lipschitz boundary.
	The following proposition clarifies an implication of these assumptions.
	This is probably a known result, but we have not found a specific reference.
	
	\begin{prop}\label{proposition of the connectedness of boundary}
		Let $\Omega\subset\Rd$ be open, bounded, with a Lipschitz boundary and such that $\Rd\setminus\partial\Omega$ has exactly two connected components.
		Then $\Omega$ and $\partial\Omega$ are connected.
	\end{prop}
	
	\begin{proof}
		Let us prove that  $U \coloneqq \Rd\setminus\overline{\Omega}$ and $\Omega$ are the two connected components of $\Rd\setminus\partial\Omega$.
		Clearly, $U$ and $\Omega$ are open, disjoint and their union is $\Rd\setminus\partial\Omega$.
		This implies that any connected set of $\Rd\setminus\partial\Omega$ is contained either in $\Omega$ or in $U$.
		
		Let $\Omega_{1}$ be a connected component of $\Omega$ and let $\Omega_{2}$ be the connected component of $\Rd\setminus\partial\Omega$ containing $\Omega_{1}$. Then $\Omega_{1}\subset\Omega_{2}\subset\Omega$ and since $\Omega_{1}$ is a connected component of $\Omega$ and $\Omega_{2}$ is connected, then $\Omega_{1}=\Omega_{2}$. This means that any connected component of $\Omega$ is also a connected component of $\Rd\setminus\partial\Omega$.
		Analogously, any connected component of $U$ is a connected component of $\Rd\setminus\partial\Omega$.
		As $\Rd\setminus\partial\Omega$ has exactly two connected components, they have to be $\O$ and $U$. 
		
		Now we show that $\Rd\setminus\Omega=\overline{U}$.
		Since $U \subset \Rd\setminus\Omega$ and $\Rd\setminus\Omega$ is closed, then $\overline{U}\subset\Rd\setminus\Omega$.
		In addition, as $\Omega$ is a Lipschitz domain, every neighborhood of every point of $\partial \Omega$ has points in $U$.
		Therefore, $\partial \Omega \subset \overline{U}$ and, hence, $\Rd\setminus\Omega = U \cup \partial \Omega \subset \overline{U}$.
		Thus, $\Rd\setminus\Omega=\overline{U}$, which is connected as the closure of a connected set.
		The result in \cite{Czarnecki} shows that $\partial\Omega$ is connected.
	\end{proof}
	
	In \cite[Theorem 4.2]{Kromer} it is proved that, when $\Rd\setminus\partial\Omega$ has exactly two connected components, for any $u\in C(\overline{\Omega};\Rd)\cap\AIB$, the function $\deg(u;\Omega; \cdot)$ is constantly $1$ or $-1$ in $\imT(u;\Omega)$.
	By Tietze's extension theorem and the fact that Brouwer's degree only depends on the boundary values, we can  give a version of that result for $u\in C(\partial\Omega;\Rd)$.
	
	\begin{teorema}\label{teorema 4.2 revisited}
		Let $\Omega\subset\Rd$ be a bounded open set such that $\Rd\setminus\partial\Omega$ has exactly two connected components and let $u\in C(\partial\Omega;\Rd)\cap\AIB$. Then there exists $\gamma\in\{\pm1\}$ such that $\deg(u;\Omega;y)=\gamma$ for every $y\in\imT(u;\Omega)$.
	\end{teorema}
	
	\section{Class $\overline{\mathcal{A}}_{p}\left(\Omega\right)$}\label{se:class}
	
	In this section we recall the functional class $\overline{\mathcal{A}}_{p}\left(\Omega\right)$, which was introduced in \cite{Henao-Mora-Oliva}.
	
	Consider a map $u\in W^{1,p}(\Omega;\Rd)$ and denote by $u _{|\p \O}$ its trace on $\p \O$.
	If $u _{|\p \O}$ belongs to $W^{1,p} (\p \O; \Rd)$, with a small abuse of notation we write $u \in W^{1,p} (\p \O; \Rd)$ and $u \in W^{1,p}(\Omega;\Rd) \cap W^{1,p} (\p \O; \Rd)$.
	The derivative of $u_{| \p \O}$ will be denoted by $D^{\tau} u$.
	That the divergence identities \eqref{eq:divergence} hold in $\O$ means that for all $\phi \in C^1_c (\O)$ and $g \in C^1_c (\Rd,\Rd)$ we have
	\begin{equation}\label{eq:dividentitiesO}
		\int_{\O} \left[ \adj D u(x) \, g(u(x)) \right] \cdot D \phi (x) \, \dd x + \int_{\O} \det D u(x) \, \phi (x) \divergencia g(u(x)) \, \dd x = 0 ,
	\end{equation}
	while that they hold in $\bar{\O}$ means that for all $\phi \in C^1 (\bar{\O})$ and $g \in C^1_c (\Rd,\Rd)$ we have
	\begin{equation}\label{eq:dividentitiesbarO}
		\begin{split}
			&\int_{\O} \left( \left[ \adj D u(x) \, g(u(x)) \right] \cdot D \phi (x) + \det D u(x) \, \phi (x) \divergencia g(u(x)) \right) \dd x \\
			& \qquad = \int_{\p \O} \phi (x) \left( \adj D^{\tau} u(x) \, g(u(x)) \right) \cdot n (x) \, \dd \mathcal{H}^{d-1} (x) ,
		\end{split}
	\end{equation}
	where $n$ is the unit exterior normal of $\p \O$.
	Clearly, if \eqref{eq:dividentitiesbarO} holds for every $\phi \in C^1 (\bar{\O})$ then \eqref{eq:dividentitiesO} holds for every $\phi \in C^1_c (\O)$.
	The geometric meaning of maps satisfying \eqref{eq:dividentitiesO} was shown in \cite{Henao-Mora, HeMo12} to be that they do not present cavitation or create new surface.
	Moreover, we know from \cite{Barchiesi} that they enjoy a great part of the regularity properties that maps in $W^{1,p}$ with $p> n$ do.
	Furthermore, the examples of \cite{Henao-Mora-Oliva} suggest that property \eqref{eq:dividentitiesbarO} implies that cavitation of $u$ is also excluded at the boundary.
	
	\begin{defi}\label{definition Ap barra}
		Let $p\geq d-1$.
		The class $\overline{\mathcal{A}}_{p}\left(\Omega\right)$ consists of those maps $u\in W^{1,p}(\Omega;\Rd)\cap W^{1,p}(\partial\Omega;\Rd)$ such that $\det Du\in L^{1}\left(\Omega\right)$ and \eqref{eq:dividentitiesbarO} holds for all $\phi\in C^{1}(\overline{\Omega})$ and $g\in C_{c}^{1}(\Rd;\Rd)$.
	\end{defi}
	
	The following result is \cite[Proposition 8.4]{Henao-Mora-Oliva} and constitutes an important step to prove injectivity of maps.
	
	\begin{prop}\label{proposition 8.4}Let $p>d-1$. If $u\in\overline{\mathcal{A}}_{p}(\Omega)$ with $\det Du\geq0$ a.e., then $\deg(u;\Omega;\cdot)=\mathcal{N}_{u}$ a.e., $\imG(u;\Omega)=\imT(u;\Omega)$ a.e.\ and $u\in L^{\infty}(\Omega;\Rd)$.\end{prop}
	
	\section{Injectivity of maps in $\overline{\mathcal{A}}_{p} \cap \AIB$}\label{se:injectivity}
	
	Our aim is to prove a refined version of the following theorem, which can be found in \cite[Theorem 9.1]{Henao-Mora-Oliva}.
	
	\begin{teorema}\label{teorema 9.1} Let $p>d-1$ and let $\Omega\subset \Rd$ be a bounded Lipschitz open set. Let $u,u_{0}\in\overline{\mathcal{A}}_{p}(\Omega)$ satisfy $u_{|\partial\Omega}=u_{0|\partial\Omega}$, $\det Du>0$ a.e., $\det Du_{0}\geq0$ a.e.\ and $u_{0}$ is injective a.e. Then $u$ is injective a.e.\ and $\imG(u;\Omega)=\imG(u_{0};\Omega)$ a.e.\end{teorema}
	
	To be precise, our main goal is to avoid the a.e.\ injectivity assumption in $\Omega$ of the boundary value $u_{0}$, replacing it with condition $\AIB$.
	
	We first state a version of the continuity of the degree.
	
	\begin{prop}\label{limite de funciones comparte grado}
		Let $\{u_{k}\}_{k\in\N}\subset C(\partial\Omega;\Rd)$ be a sequence such that $u_{k}\rightarrow u$ uniformly on $\partial\Omega$ as $k\to\infty$. Then for every $y\in\imT(u;\Omega)$ there exists some $k_{0}\in\N$ such that $\deg(u_{k};\Omega;y)=\deg(u;\Omega;y)$ for all $k\geq k_{0}$.
	\end{prop}
	
	\begin{proof}
		Let $y\in\imT(u;\Omega)$, so $y\notin u(\partial\Omega)$.
		By the uniform convergence of $\{u_{k}\}_{k\in\N}$ and the continuity of the degree, there exists some $k_{0}\in\N$ such that $y\notin u_{k}(\partial\Omega)$ for every $k\geq k_{0}$ and $\deg(u_{k};\Omega;y)=\deg(u;\Omega;y)$ for every $k\geq k_{0}$.
	\end{proof}
	
	With this, we proceed to prove the theorem.
	
	\begin{teorema}\label{teorema 9.1 revisited}
		Let $p>d-1$ and let $\Omega\subset\Rd$ be a bounded Lipschitz open set such that $\Rd\setminus\partial\Omega$ has exactly two connected components. Let $u\in\overline{\mathcal{A}}_{p}(\Omega)\cap\AIB$ and assume that $\det Du > 0$ a.e.
		Then $u$ is injective a.e.\ in $\Omega$ and $\imG(u;\Omega)=\imT(u;\Omega)$ a.e.
	\end{teorema}
	
	\begin{proof}
		Let $\{u_{k}\}_{k\in\N}\subset C(\partial\Omega;\Rd)$ be the sequence uniformly convergent to $u$ from Definition \ref{definition AIB}.
		By Theorem \ref{teorema 4.2 revisited} we have that $\deg(u_{k};\Omega;\cdot) = \gamma_k$ in $\imT(u_{k};\Omega)$ for some $\gamma_k \in\{\pm1\}$ and $k\in\N$.
		
		By Proposition \ref{proposition 8.4} and the fact that $\mathcal{N}_{u}$ is a non-negative function we can see that $\deg(u;\Omega;\cdot)\geq0$ a.e.\ in $\Rd\setminus u(\partial\Omega)$.
		In fact, by the continuity of the degree, $\deg(u;\Omega;\cdot) \geq 0$ everywhere in $\Rd\setminus u(\partial\Omega)$.
		In particular, $\deg(u;\Omega;\cdot) > 0$ everywhere in $\imT (u, \Omega)$.
		By Proposition \ref{limite de funciones comparte grado}, for every $y\in\imT(u;\Omega)$ there exists some $k_{0}\in\N$ such that $\deg(u_{k};\Omega;y)=\deg(u;\Omega;y)$ for all $k\geq k_{0}$, so $y \in \imT (u_k; \Omega)$ and $\deg(u_{k};\Omega;y)=1$.
		Therefore, 
		\[
		\deg(u;\Omega;\cdot)=1 \quad \text{in } \imT(u;\Omega) .
		\]
		By Proposition \ref{proposition 8.4} again,
		\begin{equation*}\label{N=1}
			\mathcal{N}_{u} = 1\text{\ a.e.\ in\ } \imG(u;\Omega).
		\end{equation*}
		As $\det Du > 0$ a.e., $u$ satisfies Lusin's $N^{-1}$ condition, i.e., the preimage of a subset of $\Rd$ with measure zero has measure zero (see, e.g., \cite[Remark 2.3 (b)]{Bresciani2024}).
		This implies that $u$ is injective a.e.
	\end{proof}
	
	Theorem \ref{teorema 9.1 revisited} is neither stronger nor weaker than Theorem \ref{teorema 9.1}. Indeed, Theorem \ref{teorema 9.1 revisited} does not request an a.e.\ injective map to coincide on $\partial\Omega$ with $u$ but needs for $\Rd\setminus\partial\Omega$ to have exactly two connected components.
	
	\section{Counterexample to global injectivity in $\mathcal{A}_{p} \cap \AIB$}\label{section counterexample}
	
	The family $\overline{\mathcal{A}}_{p}(\Omega)$, as opposed to $\mathcal{A}_p (\O)$ (see \cite{Barchiesi}), requires certain regularity at the boundary, as shown in \cite[Sect.\ 5]{Henao-Mora-Oliva}.
	Similarly, deformations in the family $\AIB$ enjoy some regularity at the boundary, as a limit of continuous injective mappings.
	Therefore, in the class $\overline{\mathcal{A}}_{p}(\Omega) \cap \AIB$ two regularity conditions are imposed on the boundary.
	In this section we show that both have to be assumed, in the sense that the conclusion of Theorem \ref{teorema 9.1 revisited} does not hold in the class $\mathcal{A}_{p}(\Omega)\cap\AIB$.
	
	The counterexample that we construct is a variant of \cite[Fig.\ 6]{Muller-Spector} (which was also used in \cite[Example 5.3]{Henao-Mora-Oliva}). 
	Let $\Omega=(-1,1)\times(0,1)$ be reference configuration, which is transformed under several deformations depicted in Figure \ref{contraejemplo final}.
	
	\begin{figure}[h]
		\centering
		\includegraphics[width=.625\textwidth]{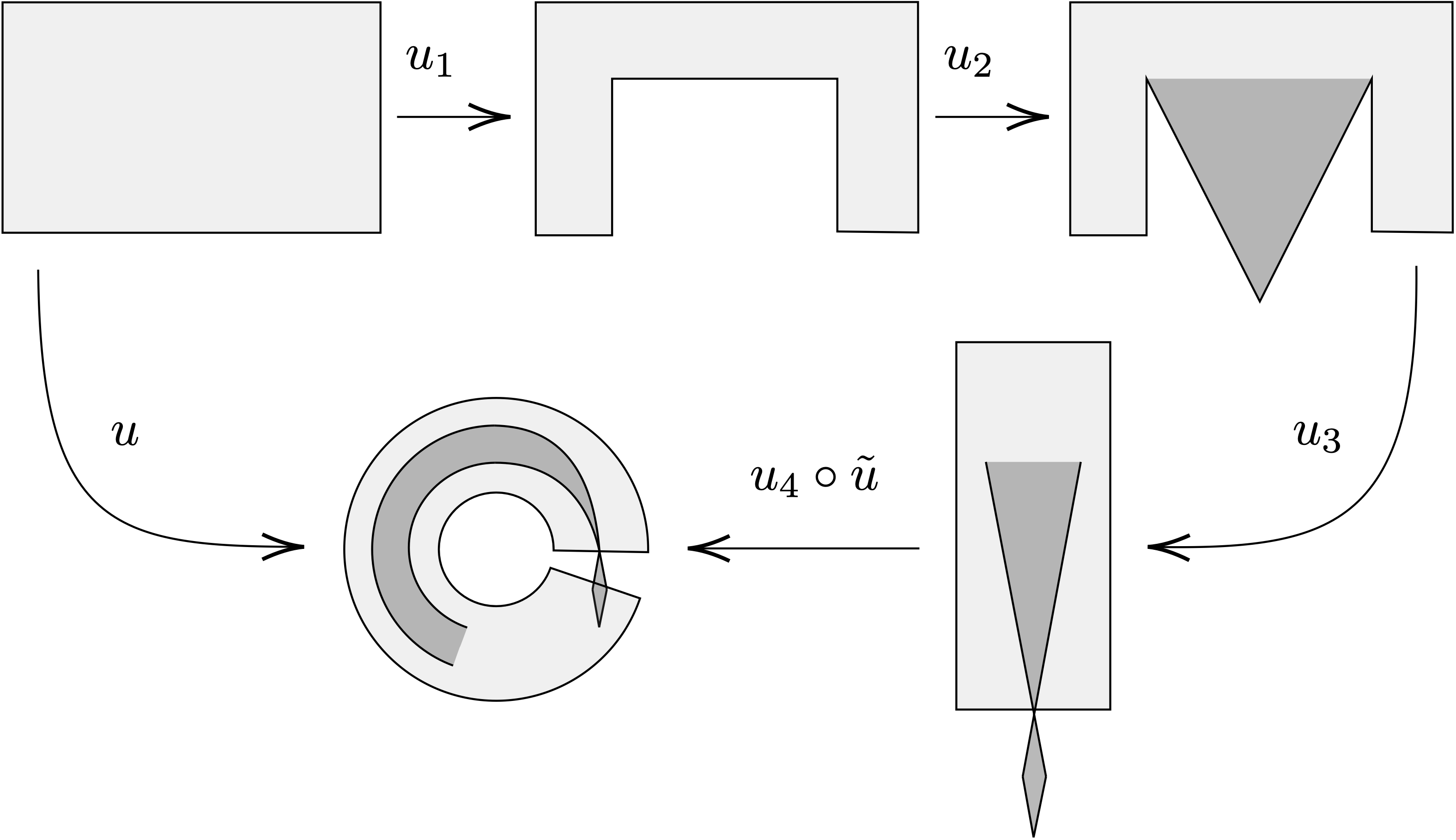}
		\caption{Counterexample to a.e.\ injectivity in $\mathcal{A}_{p}((-1,1)\times(0,1))\cap\AIB$.}
		\label{contraejemplo final}
	\end{figure}
	
	The first of these deformations is
	\begin{equation*}
		u_{1}:\Omega \rightarrow\R^{2} , \qquad u_1 (x) = \frac{\abs{x}_{\infty}+3}{4\abs{x}_{\infty}}x,
	\end{equation*}
	where $\abs{x}_{\infty}$ is the max-norm of the vector $x$.
	The map $u_1$ creates a cavity on the boundary of $\Omega$.
	The second deformation, $u_2:\R^{2}\rightarrow\R^{2}$, defined by 
	\[
	u_2(x_{1},x_{2}) =
	\begin{cases}
		(x_{1},1-(1-x_{2})(7-8\abs{x_{1}})) & \text{if}\ \abs{x_1}<\frac{3}{4},\\
		(x_{1},x_{2}) & \text{if}\ \abs{x_1}\geq\frac{3}{4},
	\end{cases}
	\]
	grips the material near the surface of the cavity and stretches it down.
	The third deformation, $u_3:\R^{2}\rightarrow\R^{2}$, closes the cavity leaving some part of the material outside the boundary and then rescales the main part of the body to fit the rectangle $[1,2]\times[-1,1]$; the leaked part of the material follows the same rescaling.
	This third deformation is defined by
	\[
	u_3(x_{1},x_{2}) =
	\begin{cases}
		(\frac{1}{2}(\sign x_{1}(1-4(1-\abs{x_{1}})(1-x_{2}))+3),2x_{2}-1) & \text{if}\ 0\leq x_{2}<\frac{3}{4},\ \frac{3}{4}<\abs{x_1},
		\\ (\frac{4x_{1}x_{2}}{4x_{2}+3}+\frac{3}{2},2x_{2}-1)& \text{if}\ \abs{x_{1}}<\frac{4x_{2}+3}{8},\ 0\leq x_{2}<\frac{3}{4},
		\\ (\frac{-4x_{1}x_{2}}{4x_{2}+3}+\frac{3}{2},2x_{2}-1)& \text{if}\ \abs{x_{1}}<\frac{4x_{2}+3}{8},\ -\frac{1}{4}\leq x_{2}<0,
		\\(\frac{x_{1}+3}{2},2x_{2}-1) & \text{elsewhere}.
	\end{cases}
	\]
	For the last map of the deformation, we first change to polar coordinates on the right half-plane $\widetilde{u}:(0, \infty) \times \R \rightarrow\R^{2}$ given by $\widetilde{u}(x_{1},x_{2})=(\sqrt{x_{1}^{2}+x_{2}^{2}}, \arctan(x_{2}/x_{1}))$ and define
	\begin{equation*}
		u_{4}: \R^{2} \rightarrow\R^{2} , \qquad u_{4}(r,\theta) = (r \cos( \alpha \theta ), r \sin ( \alpha \theta ))
	\end{equation*}
	for some $\frac{4\pi}{5}<\alpha<\pi$. The map $u_{4}\circ\widetilde{u}$ is a revolution of the ``leaking rectangle'' around the point $0\in\R^{2}$ and creates an overlaping surface between the leaked part of the material and the top part of $\Omega$ under the previous deformations. Therefore, the deformation $u=u_{4}\circ\widetilde{u}\circ u_{3}\circ u_{2}\circ u_{1}$ is not injective a.e.
	
	Arguing as in \cite[Examples 5.2 and 5.3]{Henao-Mora-Oliva}, one can show that the map $u$ is in $\mathcal{A}_p (\O)$ for any $1 \leq p < 2$, but not in $\bApO$.
	In addition, $\det D u > 0$ a.e.
	Finally, $u_{|\partial\Omega}=u_{0|\partial\Omega}$ for some diffeomorphism $u_{0}:\overline{\Omega}\rightarrow\R^{2}$, so in particular $u \in \AIB$.
	Indeed, this was shown in \cite{Muller-Spector,Henao-Mora-Oliva} for the map $u_{3}\circ u_{2}\circ u_{1|\partial\Omega}$, while the part $u_{4} \circ \widetilde{u}$ of the deformation maintains the same property.
	
	The key point allowing the loss of injectivity is that the cavitation at the boundary permits a lack of the monotonicity of the degree with respect to the domain; that is, for many open sets $U \subset \O$ it is not true that $\deg(u;\Omega;\cdot)\geq\deg(u; U;\cdot)$.
	Explicit instances of such $U$ are $(-1 + \delta, 1- \delta) \times (\delta, 1 - \delta)$ for $\delta > 0$ small.
	
	\section{Weak continuity of minors of tangential derivatives}\label{seccion algebra multi}
	
	The objective of the rest of this article is to show the existence of minimizers of an appropriate functional in the class $\overline{\mathcal{A}}_{p}(\Omega)\cap\AIB$.
	For this, we will show the weak continuity of the minors of $D^{\tau}u$ in $W^{1,p}(\partial\Omega;\Rd)$.
	The map $D^{\tau}u$ is the tangential derivative of $u \in W^{1,p}(\partial\Omega;\Rd)$, which sends $\mathcal{H}^{d-1}$-a.e.\ $x \in \partial\Omega$ to $D^{\tau}u(x) \in \mathcal{L}(T_{x}\partial\Omega;\Rd)$.
	Here, $T_{x}\partial\Omega$ is the tangent space of $\p \O$ at $x$.
	We also denote by $T\partial\Omega =\{(x,v):x\in\partial\Omega, \, v \in T_{x}\partial\Omega \}$ the tangent bundle of $\partial\Omega$, and define $T^{d}\partial\Omega\coloneqq\{(x,F):x\in\partial\Omega,F\in (T_{x}\partial\Omega)^{d}\}$.
	
	\subsection{Minors of linear maps}\label{subsec: minors fixed x}
	
	Let $V\subseteq\Rd$ be an $m$-dimensional vector space, for some number $1\leq m\leq d$, and let $1\leq k\leq m$ be an integer.
	Let $L \in \mathcal{L} (V; \Rd)$, fix a basis in $V$ and consider the matrix representation of $L$ with respect to that basis in $V$ and the canonical basis in $\Rd$.
	Given $1 \leq i_1 < \cdots < i_k \leq d$ and $1 \leq j_1 < \cdots < j_k \leq m$, we denote by
	\[
	M_{\substack{i_{1}\ldots i_{k} \\ j_{1}\ldots j_{k}}}(L)
	\]
	the minor of order $k$ resulting by the choice of rows $i_{1},\ldots,i_{k}$ and columns $j_{1},\ldots,j_{k}$ in the matrix representation of $L$.
	There are $\binom{m}{k}\binom{d}{k}$ minors of $L$ of order $k$, and $\sum_{k=1}^{m}\binom{m}{k}\binom{d}{k}$ minors of $L$ of any order. We will denote this last number by $\nu_{m}$ and we will use the convention that $\nu_{0}=0$; this notation does not indicate the dependence on $d$, since $d$ is fixed throughout the article.
	Particularly important are $\nu_{d}$, the number of minors of any $d\times d$ matrix, and $\nu_{d-1}$, the number of minors of any $d \times (d-1)$ matrix. This notation will be of use in Sections \ref{se:tangential} and \ref{sec: Interface polyconvexity}.
	
	Let $k\leq m$.
	We define $M_{k}(L)$ as the ordered sequence of all minors of order $k$ of $L$, $M_{k}^{0}(L)$ as the ordered sequence of the minors of order $k$ of $L$ not involving the last column of $L$, and $M_{k}^{1}(L)$ as the ordered sequence of the minors of order $k$ of $L$ involving the last column of $L$, all with respect to the matrix representation of $L$. Thus, 
	\begin{gather*}
		M_{k}(L) \coloneqq \left( M_{\substack{ i_{1},\ldots,i_{k} \\ j_{1},\ldots,j_{k} } }(L) \right)_{ 1\leq j_{1}<\cdots<j_{k}\leq m }^{ 1\leq i_{1}<\cdots<i_{k}\leq d } , \\
		M_{k}^{0} (L) 
		\coloneqq 
		\left( M_{\substack{ i_{1},\ldots,i_{k} \\ j_{1},\ldots,j_{k} } }(L) \right)_{ 1\leq j_{1}<\cdots<j_{k}\leq m -1}^{ 1\leq i_{1}<\cdots<i_{k}\leq d }
		\quad 
		\text{and}
		\quad
		M_{k}^{1} (L) 
		\coloneqq 
		\left( M_{\substack{ i_{1},\ldots,i_{k} \\ j_{1},\ldots,j_{k-1}, m } }(L) \right)_{ 1\leq j_{1}<\cdots<j_{k-1}\leq m -1}^{ 1\leq i_{1}<\cdots<i_{k}\leq d }.
	\end{gather*}
	Moreover, let $\{v_{1},\ldots,v_{d}\}$ be a basis of $\Rd$ such that $\{v_{1},\ldots,v_{d-1}\}$ is a basis of some subspace $W\subset\Rd$ and let $F\in\mathcal{L}(\Rd;\Rd)$; then,
	\begin{equation}\label{eq:M0kMk}
		M^0_k (F) = M_k (F_{|W}). 
	\end{equation}
	
	Finally, we define the sequence of all minors of any order of $L$ as
	\begin{equation}\label{eq: string of minors ordered}
		M(L) \coloneqq \left( M_{1}^{0}(L), \ldots, M_{m}^{0}(L), M_{1}^{1}(L), \ldots, M_{m}^{1}(L) \right) ,
	\end{equation}
	with the same convention. For the sake of notation, we also define 
	\begin{equation}\label{eq:M0M1}
		M^{0}(L) \coloneqq \left( M_{1}^{0}(L), \ldots, M_{m}^{0}(L) \right)
		\qquad
		\text{and}
		\qquad
		M^{1}(L) \coloneqq \left( M_{1}^{1}(L), \ldots, M_{m}^{1}(L) \right).
	\end{equation}
	Following the previous notation, we have that $M (L) \in \R^{\nu_m}$.
	Moreover, when $m = d$, the last component of $M (L)$ is $\det(L)$.
	
	We will use the same notation for the minors if $L$ is a given matrix instead of a linear map.
	
	\subsection{Convergence of minors of tangential derivatives}\label{subsec: minors moving x}
	
	\begin{defi}\label{defi: measurable basis}
		We say that $\{ v_1 , \ldots, v_{d-1} \}$ is a measurable basis of $T\partial\Omega$ if $v_i : \p \O \to \Rd$, for $i = 1, \ldots, d-1$, is a measurable map and $\{ v_1 (x), \ldots, v_{d-1} (x) \}$ is a basis of $T_x {\p \O}$ for $\mathcal{H}^{d-1}$-a.e.\ $x \in \partial\Omega$.
		The measurable basis is called orthonormal if so is $\{ v_1 (x), \ldots, v_{d-1} (x) \}$ for $\mathcal{H}^{d-1}$-a.e.\ $x \in \partial\Omega$.
	\end{defi}
	
	When such basis is fixed we can consider $D^{\tau} u$ as a map from $\partial \Omega$ to $\R^{d\times(d-1)}$, and $M(D^{\tau} u)$ as a map from $\partial \Omega$ to $\R^{\nu_{d-1}}$.
	Moreover, we can choose the map $n : \p \O \to S^{d-1}$ defined as
	\[
	n(x) = \frac{ v_{1}(x)\wedge\cdots\wedge v_{d-1}(x)}{\|v_{1}(x)\wedge\cdots\wedge v_{d-1}(x)\|}
	\] 
	such that the vector $n(x)$ is the outward normal to $\O$ at $x$ and
	\begin{equation}\label{eq:vi3}
		\{ v_1 (x), \ldots, v_{d-1} (x) , n(x) \} \text{ is a basis of } \R^d.
	\end{equation}
	
	The following observation calculates the minors $M^1$ of a type of maps relevant in Sections \ref{se:tangential} and \ref{sec: Interface polyconvexity}.
	
	\begin{remark}\label{remark: existence of bases}
		Consider the basis \eqref{eq:vi3} of $\R^d$.
		If $L \in \mathcal{L} (\Rd; \Rd)$ satisfies $L n(x)=0$ then $M^{1}(L)=0$.
	\end{remark}

	\begin{defi}\label{def: L infinity basis}
		Let $\mathcal{V} = \{ v_{1},\ldots,v_{d-1} \}$ be a measurable basis of $T\partial\Omega$, for $\mathcal{H}^{d-1}$-a.e.\ $x \in \partial\Omega$ let $P_{x} : \R^{d-1} \rightarrow T_{x}\partial\Omega$ with $P_{x}e_{i}=v_{i}(x)$ for each $e_{i}$ in the canonical basis of $\R^{d-1}$. We say that $\mathcal{V}$ is an $L^{\infty}$ basis of $T\partial\Omega$ if there exists $\tilde{P}_{x}:\Rd \rightarrow \Rd$ a linear extension of $P_{x}$ such that $\tilde{P}_{x} , \tilde{P}_{x}^{-1} \in L^{\infty}(\partial\Omega ; \mathcal{L}(\Rd;\Rd))$ for $\mathcal{H}^{d-1}$-a.e.\ $x \in \partial\Omega$.
	\end{defi}
	
	We will use $\mathcal{V} = \{ v_{1},\ldots,v_{d-1} \}$ to refer to a basis of $T\partial\Omega$ and $\mathcal{V}_{x} = \{ v_{1}(x),\ldots,v_{d-1}(x) \}$ for a given $x\in\partial\Omega$ with the subindex notation, to refer to the associated basis of $T_{x}\partial\Omega$.
	
	We introduce the notation regarding the parametrization of $\p \Omega$ (see, e.g., \cite[Sect.\ 3]{Henao-Mora-Oliva}).
	Let $\pi:\Rd\rightarrow\R^{d-1}$ be projection on the first $d-1$ coordinates, and $\eta:\R^{d-1}\rightarrow\Rd$ the function $\hat{z}\mapsto(\hat{z},0)$.
	As $\O$ is a Lipschitz domain, there exist $r,\beta>0$, an integer $m_{0}\geq 1$ and bi-Lipschitz maps
	\begin{equation*}
		G_{i}:[0,r]^{d-1}\times[-\beta,\beta] \rightarrow\Rd , \qquad i \in \{ 1, \ldots , m_0 \}
	\end{equation*}
	such that, when one defines $\Gamma_{i} = G_i ((0, r)^{d-1} \times \{ 0 \})$, we have that $\{ \Gamma_{i} \}_{i=1}^{m_0}$ is an open cover of $\partial\Omega$.
	For each $i \in \{ 1, \ldots , m_0 \}$ we define the bi-Lipschitz map $\Psi_{i} \coloneqq G_{i}\circ \eta : [0,r]^{d-1} \rightarrow \Gamma_{i}$.
	For $\hat{z}\in[0,r]^{d-1}$, we consider the matrix representation of $D\Psi_{i}(\hat{z})$, with columns $D\Psi_{i}^{(j)}(\hat{z}) \in \Rd$ for $j=1, \ldots, d-1$, and the basis
	$
	\mathcal{B}_{\Psi_{i}(\hat{z})} \coloneqq \{D\Psi_{i}^{(1)}(\hat{z}),\ldots, D\Psi_{i}^{(d-1)}(\hat{z}) \}
	$
	of $T_{\Psi_{i}(\hat{z})}\Gamma_{i}$.
	We will use the notation $\mathcal{B}_{\Psi_{i}}$ whenever we use the basis $\mathcal{B}_{\Psi_{i}(\hat{z})}$ for every $\hat{z}\in [0,r]^{d-1}$.
	For any $u : \partial \Omega \to \R^d$, the functions
	\begin{equation*}
		\mathcal{L}_{i}(u):\pi(G_{i}^{-1}(\Gamma_{i})) \rightarrow \R^d , \quad \mathcal{L}_{i}(u) :=  u\circ \Psi_{i} , \qquad i \in \{ 1, \ldots , m_0 \}
	\end{equation*}
	satisfy the following property (see \cite[Lemma 3.3]{Henao-Mora-Oliva}).
	
	\begin{lema}\label{Lema de if and only if}
		Let $p\geq1$. For each $n\in\N$,
		\begin{enumerate}[label=(\roman*)]
			\item\label{Lema de if and only if (i)} let $u_{n},u\in W^{1,p}(\partial\Omega;\Rd)$. Then $u_{n}\rightharpoonup u$ in $W^{1,p}(\partial\Omega;\Rd)$ as $n\to\infty$ if and only if $\mathcal{L}_{i}(u_{n})\rightharpoonup\mathcal{L}_{i}(u)$ in $W^{1,p}((0,r)^{d-1};\Rd)$ as $n\to\infty$ for all $i=1,\ldots,m_{0}$.
			\item\label{Lema de if and only if (ii)} let $u_{n},u\in L^{p}(\partial\Omega;\Rd)$. Then $u_{n}\rightharpoonup u$ in $L^{p}(\partial\Omega;\Rd)$ as $n\to\infty$ if and only if $\mathcal{L}_{i}(u_{n})\rightharpoonup\mathcal{L}_{i}(u)$ in $L^{p}((0,r)^{d-1};\Rd)$ as $n\to\infty$ for all $i=1,\ldots,m_{0}$.
		\end{enumerate}
	\end{lema}
	
	Although there is an intrinsic definition of the spaces $W^{1,p}(\partial\Omega;\Rd)$ and $L^{p}(\partial\Omega;\Rd)$ and their convergences, we will always use them refering to the result above. We also have the following result regarding the basis $\mathcal{B}_{\Psi_{i}}$.
	
	\begin{lema}
		$\mathcal{B}_{\Psi_{i}}$ is an $L^{\infty}$ basis of  $T\Gamma_{i}$ for each $i\in\{1,\ldots,m_{0}\}$.
		Moreover, there exists an $L^{\infty}$ basis of  $T\partial \Omega$.
	\end{lema}
	\begin{proof}
		Fix $i\in\{1,\ldots,m_{0}\}$. 
		Observe that $DG_{i}(x) : \Rd \rightarrow \Rd$ extends $D\Psi_{i}(x) : \R^{d-1} \rightarrow T_{x}\partial\Omega$ for $\mathcal{H}^{d-1}$-a.e.\ $x \in \partial\Omega$ in the sense that $DG_{i}(x)$ can be seen as a map from $\R^{d-1}\times\{0\}$ to $\Rd$. Since $G_{i}$ is a bi-Lipschitz map we have that $DG_{i} : \Rd \rightarrow \R^{d\times d}$ and its inverse $(DG_{i})^{-1} : \Rd  \rightarrow \R^{d\times d}$ are essentially bounded.
		
		To construct an $L^{\infty}$ basis of $T\partial\Omega$ we can join the bases of each $T\Gamma_{i}$ in the following way: for $x \in \Gamma_{1}$ we use the basis $\mathcal{B}_{\Psi_{1}}$, and for $x \in \Gamma_{s}\setminus\bigcup_{j=1}^{s-1}\Gamma_{j}$ for some $2 \leq s\leq m_{0}$ we use $\mathcal{B}_{\Psi_{s}}$.
	\end{proof}
	
	Unlike in Section \ref{subsec: minors fixed x}, we need to give a precise definition of the convergence of minors of a linear map without the need of fixing bases.
	
	\begin{defi}\label{def: convergencia de menores de una matriz}
		Let $\{f_{n}\}_{n\in\N}$ be a sequence of maps $f_{n} : T\partial\Omega \rightarrow \Rd$ such that $f_{n}(x,\cdot) : T_{x}\partial\Omega \rightarrow \Rd$ is linear for $\mathcal{H}^{d-1}$-a.e.\ $x\in\partial\Omega$. We say that $Ml( f_{n} ) \rightharpoonup Ml( f )$ in $L^{q}( \partial\Omega )$ for some $q\geq 1$ if there exists $\mathcal{V}$ an $L^{\infty}$ basis of $T\partial\Omega$ such that $M( f_{n} ) \rightharpoonup M( f )$ in $L^{q}( \partial\Omega;\R^{\nu_{d-1}} )$ where the matrix representation of each $f_{n}$ and $f$ is taken with respect to $\mathcal{V}$ and the canonical basis of $\Rd$.
	\end{defi}
	
	The convergence of minors of a linear map is independent of the choice of the $L^{\infty}$ basis.
	
	\begin{prop}
		Let $\mathcal{V}$ and $\mathcal{B}$ be two $L^{\infty}$ bases of $T\partial\Omega$, let $\{f_{n}\}_{n\in\N}$ be a sequence of maps $f_{n} : T\partial\Omega \rightarrow \Rd$ such that $f_{n}(x,\cdot) : T_{x}\partial\Omega \rightarrow \Rd$ is linear for $\mathcal{H}^{d-1}$-a.e.\ $x\in\partial\Omega$ and such that $M( f_{n} ) \rightharpoonup M( f )$ in $L^{q}( \partial\Omega ;\R^{\nu_{d-1}})$ for some $q\geq 1$ where the matrix representation of $f$ and $f_{n}$ are with respect to $\mathcal{V}$ and the canonical basis of $\Rd$. Then $M( f_{n} ) \rightharpoonup M( f )$ in $L^{q}( \partial\Omega ;\R^{\nu_{d-1}})$ where the matrix representation of $f$ and $f_{n}$ are with respect to $\mathcal{B}$ and the canonical basis of $\Rd$.
	\end{prop}
	\begin{proof}
		Let $\mathcal{V} = \{v_{1},\ldots,v_{d-1}\}$ and $\mathcal{B} = \{b_{1},\ldots,b_{d-1}\}$. Denote by $\mathcal{V}_{f_{n}}$ and $\mathcal{B}_{f_{n}}$ the matrix representations of $f_{n}$ with respect to $\mathcal{V}$ and $\mathcal{B}$, respectively, and the canonical basis of $\Rd$, and denote by $\mathcal{V}_{f_{n}(x)}$ and $\mathcal{B}_{f_{n}(x)}$ the matrix representations of $f_{n}(x,\cdot)$ with respect to $\mathcal{V}_{x}$ and $\mathcal{B}_{x}$, respectively, and the canonical basis of $\Rd$. For $\mathcal{H}^{d-1}$-a.e.\ $x\in\partial\Omega$ there exist measurable maps $\{a_{i,j}\}_{j=1}^{d-1}$  from $\partial\Omega$ to $\R$ such that $v_{i}(x) = \sum_{j=1}^{d-1}a_{j,i}(x)b_{j}(x)$ for each $i\in\{1,\ldots,d-1\}$, i.e., for $\mathcal{H}^{d-1}$-a.e.\ $x\in\partial\Omega$ there exists $A_{x} = (a_{j,i}(x))\in\R^{(d-1)\times (d-1)}$ such that $\mathcal{V}_{f_{n}(x)} A_{x}^{-1}= \mathcal{B}_{f_{n}(x)}$. Taking minors we obtain that $M( \mathcal{V}_{f_{n}(x)} A_{x}^{-1})= M(\mathcal{B}_{f_{n}(x)})$ and since $\mathcal{V}$ and $\mathcal{B}$ are $L^{\infty}$ bases we have that $A_{x}$ and $A_{x}^{-1}$ are bounded.
		Therefore, by the Cauchy-Binet formula, there exists a linear map $\mathfrak{F}_{x}:\R^{\nu_{d-1}} \rightarrow \R^{\nu_{d-1}}$ such that $\mathfrak{F}_{x}( M( \mathcal{V}_{f_{n}(x)} ))= M(\mathcal{B}_{f_{n}(x)})$ for $\mathcal{H}^{d-1}$-a.e.\ $x\in\partial\Omega$. In the same way, there exists a linear map $\mathfrak{F}:\R^{\nu_{d-1}} \rightarrow \R^{\nu_{d-1}}$ such that $\mathfrak{F}( M( \mathcal{V}_{f_{n}} ))= M(\mathcal{B}_{f_{n}})$ and hence, since $M( \mathcal{V}_{f_{n}} ) \rightharpoonup M( \mathcal{V}_{f} )$ in $L^{q}(\partial\Omega;\R^{\nu_{d-1}})$ we also have that $M( \mathcal{B}_{f_{n}} ) \rightharpoonup M( \mathcal{B}_{f} )$ in $L^{q}(\partial\Omega;\R^{\nu_{d-1}})$.
	\end{proof}
	
	A result on the weak continuity of minors was proved in \cite[Prop.\ 15]{BernardBessi} using geometric tools.
	We present a straightforward proof in the next proposition.
	
	\begin{prop}\label{convergencia de menores}
		Let $p > d-1$.
		Let $u\in W^{1,p}(\partial\Omega;\Rd)$ and let $\{u_{n}\}_{n\in\N}\subset W^{1,p}(\partial\Omega;\Rd)$ be such that $u_{n}\rightharpoonup u$ in $W^{1,p}(\partial\Omega;\Rd)$ as $n\to\infty$. Then $Ml(D^{\tau}u_{n})\rightharpoonup Ml(D^{\tau}u)$ in $L^{1}(\partial\Omega)$ as $n\to\infty$. 
	\end{prop}
	\begin{proof}
		By Lemma \ref{Lema de if and only if}\ref{Lema de if and only if (i)} we have that $\mathcal{L}_{i}(u_{n})\rightharpoonup\mathcal{L}_{i}(u)$ in $W^{1,p}((0,r)^{d-1};\Rd)$ for each $i\in\{1,\ldots,m_{0}\}$. For each $i\in\{1,\ldots,m_{0}\}$ the result of \cite[Theorem 8.20]{Dacorogna} gives us that $M( D \mathcal{L}_{i}(u_{n})  ) \rightharpoonup M(  D \mathcal{L}_{i}(u)  )$ in $L^{1}((0,r)^{d-1};\R^{\nu_{d-1}})$ where both matrix representations are with respect to the canonical bases.
		
		As $\mathcal{L}_{i}(u) = u\circ\Psi_{i}$ we have that $D\mathcal{L}_{i} (u)(\hat{z}) = D^{\tau}u(\Psi_{i}(\hat{z}))D\Psi_{i}(\hat{z}) : \R^{d-1}  \rightarrow  \Rd$ for each $i \in \{1,\ldots,m_{0}\}$ and any $\hat{z}\in(0,r)^{d-1}$. Fix the $L^{\infty}$ basis $\mathcal{B}_{\Psi_{i}}$ and observe that for any $\hat{z}\in(0,r)^{d-1}$ we have that $D\Psi_{i}(\hat{z}) : \R^{d-1} \rightarrow T_{\Psi_{i}(\hat{z})}\partial\Omega$ is defined by $e_{j} \mapsto D\Psi_{i}(\hat{z})e_{j} = D\Psi_{i}^{(j)} (\hat{z})$ for each $e_{j}$ in the canonical basis of $\R^{d-1}$; consequently, $D\Psi_{i} =\Id$. On the other hand, since $D\mathcal{L}_{i}(u) = D^{\tau}u(\Psi_{i})$ with respect to $\mathcal{B}_{\Psi_{i}}$ and the canonical basis, we have that 
		$
		M( D^{\tau}u_{n}(\Psi_{i}) ) \rightharpoonup M( D^{\tau}u(\Psi_{i}))
		$ in $
		L^{1}((0,r)^{d-1};\R^{\nu_{d-1}}),
		$
		where again, the matrices are with respect to $\mathcal{B}_{\Psi_{i}}$ and the canonical basis of $\R^{d-1}$.
		By Definition \ref{def: convergencia de menores de una matriz} this means that 
		\begin{equation}\label{eq: convergence of M D t u Psi }
			Ml( D^{\tau}u_{n}(\Psi_{i}) ) \rightharpoonup Ml( D^{\tau}u(\Psi_{i}) )
			\qquad
			\text{in}
			\qquad
			L^{1}((0,r)^{d-1}).
		\end{equation}
		Observe that 
		$
		Ml(  D^{\tau}u(\Psi_{i})  )
		=
		Ml(D^{\tau}u)\circ\Psi_{i}
		=
		\mathcal{L}_{i}( Ml(D^{\tau}u) )
		$ and hence, expression \eqref{eq: convergence of M D t u Psi } means that 
		\[
		\mathcal{L}_{i} (  Ml( D^{\tau}u_{n} )  ) \rightharpoonup \mathcal{L}_{i} (  Ml( D^{\tau}u )  )
		\qquad
		\text{in}
		\qquad
		L^{1}((0,r)^{d-1};\R^{\nu_{d-1}}).
		\]
		Lemma \ref{Lema de if and only if}\ref{Lema de if and only if (ii)} gives us that $Ml( D^{\tau}u_{n} ) \rightharpoonup Ml( D^{\tau}u )$ in $L^{1}( \partial\Omega)$.
	\end{proof}
	
	\section{Tangential polyconvexity and quasiconvexity}\label{se:tangential}
	
	We first give a definition used along the rest of the article (see \cite{Dacorogna}).
	
	\begin{defi}\label{polyconvexity} Let $V\subset\Rd$ be a $m$-dimensional vector space for some natural $m\leq d$.
		\begin{enumerate}[label=(\roman*)]
			\item\label{def polyconvexity (i)}
			A function $f:\R^{d\times m}\rightarrow \R$ is said to be polyconvex if there exists $\varphi:\R^{\nu_{m}} \rightarrow \R$ convex such that $f(F) = \varphi(M(F))$.
			
			\item\label{def polyconvexity (ii)}
			A function $W_{0}:\mathcal{L}(V;\Rd)\rightarrow\R$ is called polyconvex if there exist $\mathcal{B}_{V}$ a measurable basis of $V$ and a convex function $\Phi:\R^{\nu_{m}}\rightarrow\R$ such that $W_{0}(F)=\Phi(M(F))$ for all $F\in\mathcal{L}(V;\Rd)$ in the sense of \ref{def polyconvexity (i)} where $M(F)$ refers to the minors of the matrix representation of $F$ with respect to $\mathcal{B}_{V}$ and the canonical basis in $\Rd$.
		\end{enumerate}
	\end{defi}
	
	We will use the cases $m=d$ and $m=d-1$.
	
	We now define the energy functional for which we will prove the existence of minimizers in $\overline{\mathcal{A}}_{p}(\Omega)\cap\AIB$.
	As natural in the theory of nonlinear elasticity, the functional will be of the form 
	\begin{equation}\label{functional to minimize}
		I[u]=\int_{\Omega}W(x,u(x),Du(x))\dd x + \int_{\partial\Omega} U(x,u(x),D^{\tau}u(x),n(x)) \dd\mathcal{H}^{d-1}(x),
	\end{equation}
	where the function $W$ refers to the elastic energy of the deformation $u$ applied on the body occupying $\Omega$ in its reference configuration, and $U$ refers to the elastic energy of the deformation $u$ applied to the boundary of the body.
	The potentials $W$ and $U$ do not usually depend on $u(x)$, but we have included them here since the theory applies also for this case.
	In fact, external forces depend on $u(x)$.
	Recall that $D^{\tau}u$ is the tangential derivative of $u_{|_{\p \O}}$.
	Proofs of these kind often only take into account the functional over $\Omega$, and follow standard polyconvexity and lower semicontinuity reasonings.
	However, as we are working in the class $\AIB$, we also need the term over $\partial\Omega$, and, hence, an analogous concept to polyconvexity on the boundary.
	
	\begin{remark}\label{remark: dominio de U}The domain of $W$ is $\Omega\times\Rd\times\R_{+}^{d\times d}$, however, the functional $U$ has a more specific domain:
		$$\mathcal{D}_{U}\coloneqq\{(x,y,F,n):x\in\partial\Omega,\ y\in\Rd,\ F\in\mathcal{L}(T_{x}\partial\Omega;\Rd),\ n\in N_{x}\partial\Omega \cap S^{d-1} \}.$$
		Note that $\mathcal{L}(T_{x}\partial\Omega;\Rd)\simeq(T_{x}\partial\Omega)^{d}$.\end{remark}

	In order to prove the existence of minimizers of $I$ we need to prove weak lower semicontinuity on the boundary integral of \eqref{functional to minimize}. In the same way that polyconvexity is sufficient for semicontinuity on the integral over $\Omega$, the following concept will provide a sufficient condition for semicontinuity on $\partial\Omega$.
	
	\begin{defi}\label{defi tangential polyconvexity}
		A function $U:T^{d}\partial\Omega\rightarrow\R$ is said to be tangentially polyconvex if there exists
		a measurable basis of $T \partial\Omega$ and a function $\Phi:\partial\Omega\times\R^{\nu_{d-1}}\rightarrow \R$ such that $\Phi(x,\cdot)$ is convex for $\mathcal{H}^{d-1}$-a.e.\ $x\in\partial\Omega$ and $U(x,F)=\Phi(x,M(F))$ for every $F\in \mathcal{L} (T_{x}\partial\Omega)^{d}$.\end{defi}
	
	The definition of tangential polyconvexity is independent of the choice of the measurable basis.
	
	\begin{prop}\label{prop: indpendence of measurable basis}
		Let $U:T^{d}\partial\Omega\rightarrow\R$ be tangentially polyconvex and let $\mathcal{V} = \{ \tilde{v}_{1},\ldots,\tilde{v}_{d-1} \}$ be a measurable basis of $T\partial\Omega$. Then there exists $\Phi_{ \mathcal{V} }:\partial\Omega\times\R^{\nu_{d-1}}\rightarrow \R$ such that $\Phi_{ \mathcal{V} }(x,\cdot)$ is convex for $\mathcal{H}^{d-1}$-a.e.\ $x\in\partial\Omega$ and such that $U(x,F)=\Phi_{ \mathcal{V} }(x,M(F))$ for every $F\in \mathcal{L} (T_{x}\partial\Omega; \R^{d})$, where the matrix representation of $F$ is with respect to $\mathcal{V}$ and the canonical basis.
	\end{prop}
	\begin{proof}
		There exist a measurable basis $\mathcal{B} \coloneqq \{ v_{1},\ldots,v_{d-1} \}$ of $T\partial\Omega$ and a map $\Phi: \partial\Omega \times \R^{\nu_{d-1}} \rightarrow \R$ such that $\Phi(x,\cdot)$ is convex for $\mathcal{H}^{d-1}$-a.e.\ $x\in\partial\Omega$ and $U(x,F) = \Phi(x,M( \mathcal{B}_{F }))$ for every $F\in\mathcal{L}(T_{x}\partial\Omega;\Rd)$ where $\mathcal{B}_{F }$ refers to the matrix representation of $F$ with respect to $\{ v_{1}(x),\ldots,v_{d-1}(x) \}$ and the canonical basis of $\Rd$. 
		Let $\mathcal{V}_{F}$ be the matrix representation of $F$ with respect to $\{ \tilde{v}_{1}(x),\ldots,\tilde{v}_{d-1}(x) \}$ and the canonical basis of $\Rd$, there exist measurable maps $\{a_{i,j}\}_{j=1}^{d-1}$ from $\partial\Omega$ to $\R$ such that $\tilde{v}_{i}(x) = \sum_{j=1}^{d-1}a_{j,i}(x)v_{j}(x)$ and therefore that
		there also exists a matrix $A_{x}=(a_{i,j}(x))_{i,j}\in\R^{(d-1)\times (d-1)}$ such that $\mathcal{V}_{F} = \mathcal{B}_{F}A_{x}^{\T}$. 
		Taking minors we obtain that $M(\mathcal{B}_{F }) = M(\mathcal{V}_{F }A^{-\T})$, and by the Cauchy-Binet formula there exists a linear map $\mathfrak{F}_{A_{x}}:\R^{\nu_{d-1}}\rightarrow\R^{\nu_{d-1}}$ such that $M( \mathcal{B}_{F }) = \mathfrak{F}_{A_{x}}(M( \mathcal{V}_{F }))$. As the composition of a linear map with a convex map is convex, we have that 
		\[
		U(x,F)
		=
		\Phi(x,M( \mathcal{B}_{F }))
		=
		\Phi( x,\mathfrak{F}_{A_{x}}(M(\mathcal{V}_{F})) )
		=
		\Phi_{\mathcal{V}}( x,M(\mathcal{V}_{F}) ) 
		\]
		for some convex map $\Phi_{\mathcal{V}} : \partial \Omega \times \R^{\nu_{d-1}} \rightarrow \R$ defined by $(x,(a_{1},\ldots,a_{\nu_{d-1}})) \mapsto \Phi_{\mathcal{B}}( x,\mathfrak{F}_{A_{x}}((a_{1},\ldots,a_{\nu_{d-1}})) )$.
	\end{proof}
	
	The relationship between tangential polyconvexity and usual polyconvexity is presented in the following proposition.
	
	\begin{prop}\label{prop: properties of tangential polyconvexity} The following are properties of tangential polyconvexity.
		\begin{enumerate}[label=(\roman*)]
			\item\label{i de la relacion de pc y pctg} Let $\widetilde{U} : \partial \Omega \times \mathcal{L}(\Rd;\Rd) \rightarrow \R$. The map $\mathcal{L}(\Rd;\Rd) \ni A\mapsto\widetilde{U}(x,A_{|T_{x}\partial\Omega})$ is polyconvex for $\mathcal{H}^{d-1}$-a.e.\ $x\in\partial\Omega$ if and only if the map $U:T^{d}\partial\Omega\rightarrow\R$ defined as $U\coloneqq\widetilde{U}_{|T^{d}\partial\Omega}$ is tangentially polyconvex.
			
			\item\label{ii de la relacion de pc y pctg} Let $U:\partial\Omega \times \mathcal{L}(\Rd;\Rd)  \rightarrow \R$ be such that $U(x,\cdot)$ is polyconvex for $\mathcal{H}^{d-1}$-a.e.\ $x\in\partial\Omega$, then $U_{|T^{d}\partial\Omega}$ is tangentially polyconvex.
		\end{enumerate}
	\end{prop}
	\begin{proof}
		\ref{i de la relacion de pc y pctg} Assume that the map $U:T^{d}\partial\Omega\rightarrow\R$ is tangentially polyconvex.
		Then there exists some $\Phi:\partial\Omega\times\R^{\nu_{d-1}}\rightarrow\R$ such that $\Phi(x;\cdot)$ is convex for $\mathcal{H}^{d-1}$-a.e.\ $x\in\partial\Omega$ and $U(x,F)=\Phi(x,M(F))$ for all $F\in(T_{x}\partial\Omega)^{d}$ with respect to a measurable basis in $T_{x}\partial\Omega$. We define
		\begin{equation*}\label{def Phi tilde}
			\begin{split}
				\widetilde{\Phi}:\partial\Omega\times\R^{\nu_{d}}&\rightarrow\R\\
				(x,(a_{1},\ldots,a_{\nu_{d}}))&\mapsto \Phi(x,(a_{1},\ldots,a_{\nu_{d-1}})).
			\end{split}
		\end{equation*}
		By \eqref{eq:M0kMk} and \eqref{eq:M0M1}, for any $\tilde{F}\in\mathcal{L}(\Rd;\Rd)$ extending $F$ we have that $M^{0}(\tilde{F}) = M(F)$.
		Moreover, by \eqref{eq: string of minors ordered}, $M(\tilde{F}) = (M(F),M_{1}^{1}(\tilde{F}), \ldots, M_{d}^{1}(\tilde{F}))$. As a consequence, we have that $\Phi(x,M(F)) = \widetilde{\Phi}(x,M(\tilde{F}))$ for $\mathcal{H}^{d-1}$-a.e.\ $x\in\partial\Omega$ and therefore, that
		$
		\widetilde{U}(x,\tilde{F})
		=
		U(x,F).
		$
		Finally $\widetilde{U}(x,\tilde{F}) = \widetilde{\Phi}(x,M(\tilde{F}))$, and since $\widetilde{\Phi}(x;\cdot)$ is convex for $\mathcal{H}^{d-1}$-a.e.\ $x\in\partial\Omega$, we have that $\widetilde{U}(x,\cdot)$ is polyconvex for $\mathcal{H}^{d-1}$-a.e.\ $x\in\partial\Omega$.
		
		Conversely, assume that the map $\mathcal{L}(\Rd;\Rd) \ni A\mapsto\widetilde{U}(x,A_{|T_{x}\partial\Omega})$ is polyconvex for $\mathcal{H}^{d-1}$-a.e.\ $x\in\partial\Omega$, then there exists some $\widetilde{\Phi} : \partial\Omega\times\R^{\nu_{d}} \rightarrow \R$ such that $\widetilde{\Phi}(x,\cdot)$ is convex for $\mathcal{H}^{d-1}$-a.e.\ $x\in\partial\Omega$ and such that $\widetilde{\Phi}(x,M(\widetilde{ F })) = \widetilde{U}(x,\widetilde{ F}_{|T_{x}\partial\Omega})$ for every $\widetilde{ F} \in \mathcal{L}(\Rd;\Rd)$. We define
		\begin{align*}
			\Phi:\partial\Omega\times\R^{\nu_{d-1}}&\rightarrow\R\\
			(x,(a_{1},\ldots,a_{\nu_{d-1}}))&\mapsto \widetilde{\Phi}(x,(a_{1},\ldots,a_{\nu_{d-1}},0,\ldots,0)),
		\end{align*}
		which is convex.
		For any $\widetilde{F}\in\mathcal{L}(\Rd;\Rd)$ we define $F_{x} \in\mathcal{L}(\Rd;\Rd)$ as the extension of $\widetilde{F}_{|T_{x}\partial\Omega}$ such that $F_{x}n(x)=0$.
		With the basis selected as in \eqref{eq:vi3} and by \eqref{eq:M0kMk} and \eqref{eq:M0M1},  we have that $M(\widetilde{F}_{|T_{x}\partial\Omega}) = M^0(F_{x})$.
		By Remark \ref{remark: existence of bases} we also have that $M^1(F_{x}) = 0\in\R^{\nu_{d}-\nu_{d-1}}$.
		Recalling \eqref{eq: string of minors ordered} we obtain that $M(F_x) = (M(\widetilde{F}_{|T_{x}\partial\Omega}),0, \ldots, 0)$.
		Therefore $\Phi( x,M(\widetilde{F}_{|T_{x}\partial\Omega}) ) = \widetilde{\Phi}( x,M(F_{x}) )$ and if we fix $\mathcal{H}^{d-1}$-a.e.\ $x\in\partial\Omega$, we have that
		$
		\widetilde{U}(x,F_{x|T_{x}\partial\Omega})
		=
		\widetilde{U}(x,\widetilde{F}_{|T_{x}\partial\Omega})
		=
		\widetilde{U}_{|T^{d}\partial\Omega}(x,\widetilde{F}_{|T_{x}\partial\Omega})
		=
		U(x,\widetilde{F}_{|T_{x}\partial\Omega}).
		$
		This leads to $\Phi( x,M(\widetilde{F}_{|T_{x}\partial\Omega}) ) = U(x,\widetilde{F}_{|T_{x}\partial\Omega})$ and
		since $\Phi(x,\cdot)$ is convex for $\mathcal{H}^{d-1}$-a.e.\ $x\in\partial\Omega$ then $U$ is tangentially polyconvex.
		
		\ref{ii de la relacion de pc y pctg} Let $n : \partial\Omega \rightarrow \Rd$ be the unit outward normal vector to $\Omega$. Since $U$ is polyconvex, for $\mathcal{H}^{d-1}$-a.e.\ $x\in\partial\Omega$ the map $\mathcal{L}(\Rd;\Rd) \ni A\mapsto U(x,A_{|T_{x}\partial\Omega})$ satisfies that there exist a measurable basis $\mathcal{V}$ and $\Phi:\partial\Omega\times\R^{\nu_{d-1}}\rightarrow\R$ such that $\Phi(x,\cdot)$ is convex for $\mathcal{H}^{d-1}$-a.e.\ $x\in\partial\Omega$ and $U(x,A_{|T_{x}\partial\Omega}) = \Phi(x,M(A_{|T_{x}\partial\Omega}))$ for each $A \in \mathcal{L}(\Rd;\Rd) $, where the matrix is taken with respect to $\mathcal{V}$ and the canonical basis of $\Rd$.
		In particular, if for $\mathcal{H}^{d-1}$-each $(x,F)\in T^{d}\partial\Omega$ we denote by $F_{x}\in\mathcal{L}(\Rd;\Rd)$ the extension of $F$ such that $F_{x}n(x)=0$, the map $U_{|T^{d}\partial\Omega}:T^{d}\partial\Omega\rightarrow\R$ satisfies that 
		$
		U_{|T^{d}\partial\Omega}(x,F) 
		=
		U(x,F_{x|T_{x}\partial\Omega}) 
		=
		\Phi(x,M(F_{x|T_{x}\partial\Omega})) 
		= 
		\Phi(x,M(F))
		$ and therefore $U_{|T^{d}\partial\Omega}$ is tangentially polyconvex. 
	\end{proof}
	
	We present the definition of tangentially quasiconvex, due to \cite{DaFoMaTr99} (see also \cite{Alicandro-Leone} and \cite{Mora-Oliva}).
	
	\begin{defi}\label{cuasiconvexidad tangencial}Let $U:T^{d}\partial\Omega\rightarrow[0,\infty)$ be a Borel function. We say that $U$ is tangentially quasiconvex if for all $(z,\xi)\in T^{d}\partial\Omega$ and all $\varphi\in W^{1,\infty}(B(0;1);T_{z}\partial\Omega)$ such that $\varphi(y)=\xi y$ on $\partial B(0;1)$ we have that
		\[
		U(z,\xi)\leq\frac{1}{\abs{B(0;1)}}\int_{B(0;1)}U(z,D\varphi(y))\dd y.
		\]\end{defi}
	
	Here $B(0;1)$ is the unit ball in $\Rd$. We are regarding $\xi$ as an $d\times d$ matrix and note that the fact $\varphi\in W^{1,\infty}(B(0;1);T_{z}\partial\Omega)$ implies that $D\varphi(x)\in(T_{z}\partial\Omega)^{d}$ for a.e.\ $x\in B(0,1)$.
	In the same way that polyconvexity is sufficient for quasiconvexity (see e.g.\ \cite{Dacorogna}), the same result holds for their tangential versions.
	
	\begin{prop}
		Let $U:T^{d}\partial\Omega\rightarrow\R$ be tangentially polyconvex. Then $U$ is tangentially quasiconvex.
	\end{prop}
	\begin{proof}
		Let $\varphi$ be as in Definition \ref{cuasiconvexidad tangencial}. Let $\Phi:\partial\Omega\times\R^{\nu_{d}}\rightarrow\R$ be such that $\Phi(x,\cdot)$ is convex and $U(x,\xi) = \Phi (x,M(\xi))$.
		Let $B=B(0;1)$.
		Then, by Jensen's inequality, for any $(x,\xi)\in T^{d}\partial\Omega$, 
		\begin{equation*}
			\frac{1}{\abs{B}}\int_{B}U(x,D\varphi(y))\dd y =\frac{1}{\abs{B}}\int_{B}\Phi(x,M(D\varphi(y)))\dd y \geq\Phi\left(x,\frac{1}{\abs{B}}\int_{B}M(D\varphi(y))\dd y\right) .
		\end{equation*}
		Now, by standard properties of minors (see, e.g., \cite[Lemma 5.5]{Dacorogna}),
		\[
		\frac{1}{\abs{B}}\int_{B}M(D\varphi(y))\dd y = \frac{1}{\abs{B}}\int_{B}M(\xi)\dd y = M(\xi) ,
		\]
		and hence,
		\[
		\frac{1}{\abs{B}}\int_{B}U(x,D\varphi(y))\dd y \geq \Phi(x,M(\xi)) = U(x,\xi),
		\]
		so proving that $U$ is tangentially quasiconvex.
	\end{proof}
	
	\section{Interface polyconvexity}\label{sec: Interface polyconvexity}
	
	Given the formulation of the tangential polyconvexity in Definition \ref{defi tangential polyconvexity}, we ought to mention the \textit{interface polyconvexity}, a similar concept developed in \cite{Silhavy}. 
	Since the notion of interface polyconvexity is not really used in this article, this section can be skipped in a first reading.
	Arising in parallel conditions, both notions respond to the need of a convexity property in the stored-energy function for surfaces. 
	While our formulation of tangential polyconvexity considers $T^d \partial\Omega$, the interface polyconvexity is defined for a given $x\in\partial\Omega$ (see \cite[Definitions 5.1 and 6.3]{Silhavy}). 
	
	We first state some definitions and facts from multilinear algebra to be used along the rest of the section.
	For $k \in \N$, the space $\Lambda_{k} \Rd$ consists of all alternating $k$-tensors on $\Rd$, i.e., sums of elements of the form $a_{1} \wedge \cdots \wedge a_{k}$ with $a_{1},\ldots,a_{k}\in \Rd$.
	Here, $\wedge$ denotes the exterior product between vectors in $\Rd$.
	We will make the natural identifications $\Lambda_{0}\Rd \simeq \Lambda_{d}\Rd \simeq \R$ and $\Lambda_{1} \Rd \simeq \Lambda_{d-1}\Rd \simeq \Rd$.
	We will repeatedly use that if $\mathcal{V} = \{ v_1 , \ldots, v_d\}$ is a basis of $\R^d$ then $\mathcal{V}_k := \{ v_{i_1} \wedge \cdots \wedge v_{i_k} : 1 \leq i_1 < \cdots < i_k \leq d \}$ is a basis of $\Lambda_k \Rd$.
	
	Let $L \in \mathcal{L} (\Rd ; \Rd)$.
	The map $\Lambda_{k}L : \Lambda_{k}\Rd \rightarrow \Lambda_{k}\Rd$ is defined as the only linear map such that $(\Lambda_{k}L)(a_{1}\wedge \cdots \wedge a_{k}) = La_{1} \wedge \cdots \wedge La_{k}$ for $a_{1},\ldots,a_{k}\in \Rd$; in particular, the map $\Lambda_{0}L$ is identified with the identity (i.e., multiplication by $1$).
	
	The next definition is from \cite[Section 1.7.5]{Federer}.
	
	\begin{defi}\label{def: inner product}
		Let $m\in\N$ and let $\mathcal{P}_{m}$ be the set of permutations of $(1,\ldots,m)$.
		The inner product in $\Lambda_{m} \Rd$, denoted by $\cdot$, is the only bilinear form such that for all $\xi_{1} , \ldots , \xi_{m} , \eta_{1} , \ldots , \eta_{m} \in \Rd$, 
		\begin{equation*}
			(\xi_{1} \wedge \cdots \wedge \xi_{m} ) \cdot ( \eta_{1} \wedge \cdots \wedge \eta_{m})  = \sum_{\sigma \in \mathcal{P}_{m}} \sign\sigma \prod_{i=1}^{m}\xi_{\sigma(i)} \cdot \eta_{i},
		\end{equation*}
		where the inner product in the right-hand side refers to the standard inner product in $\Rd$.
	\end{defi}
	
	The following result describes the inner product defined above acting on an orthonormal basis.
	
	\begin{lema}\label{lemma inner product on basis}
		Let $\mathcal{V}$ be an orthonormal basis of $\Rd$, let $\xi_{1}, \ldots, \xi_{m}$ be $m$ different elements of $\mathcal{V}$, let $\eta_{1}, \ldots, \eta_{m}$ be $m$ different elements of $\mathcal{V}$ and let $\xi = \xi_{1} \wedge \cdots \wedge \xi_{m}$ and $\eta = \eta_{1} \wedge \cdots \wedge \eta_{m}$.
		
		\begin{enumerate}[label=(\roman*)]
			\item\label{lemma inner product on basis i} If $\{ \xi_{1},\ldots,\xi_{m} \} \neq \{ \eta_{1},\ldots,\eta_{m} \}$ then $\xi\cdot\eta=0$.
			
			\item\label{lemma inner product on basis ii} If $\{ \xi_{1},\ldots,\xi_{m} \} = \{ \eta_{1},\ldots,\eta_{m} \}$ then 
			$
			\xi \cdot \eta = \sign \tilde{\sigma}
			$,
			where $\tilde{\sigma}$ is the only permutation such that $\xi_{\tilde{\sigma}(i)} = \eta_{i}$ for all $i \in \{ 1, \ldots, m\}$.
		\end{enumerate}
	\end{lema}
	\begin{proof} \ref{lemma inner product on basis i} For each $\sigma \in \mathcal{P}_{m}$ we have that $\xi_{\sigma(i)} \neq \eta_{i}$ for some $i \in \{1, \ldots, m\}$, so $\xi_{\sigma(i)} \cdot \eta_{i} = 0$.
		Consequently, $\prod_{i=1}^{m}\xi_{\sigma(i)} \cdot \eta_{i} =0$.
		
		\ref{lemma inner product on basis ii} For each $\sigma \in \mathcal{P}_{m} \setminus \{ \tilde{\sigma} \}$ we have that $\prod_{i=1}^{m}\xi_{\sigma(i)} \cdot \eta_{i} =0$, as in \ref{lemma inner product on basis i}.
		Therefore, $\xi \cdot \eta = \sign \tilde{\sigma} \prod_{i=1}^{m}\xi_{\tilde{\sigma}(i)} \cdot \eta_{i} = \sign \tilde{\sigma} \prod_{i=1}^{m} \eta_i \cdot \eta_{i} = \sign \tilde{\sigma}$.
	\end{proof}
	
	As a consequence of Lemma \ref{lemma inner product on basis}\ref{lemma inner product on basis ii}, when $m=1$, the product of Definition \ref{def: inner product} is the standard inner product in $\Lambda_{1}\Rd \simeq \Rd$, and when $m=0$, it is the product of real numbers in $\Lambda_0 \Rd \simeq \R$.
	
	The next definition is from \cite[Appendix C]{Silhavy}.
	
	\begin{defi}\label{def: contraction}
		Let $0\leq s\leq r$ be natural numbers. Let $\alpha \in \Lambda_{r}\Rd$ and $\beta \in \Lambda_{s}\Rd$. We define the contraction $\alpha \Contraction \beta \in \Lambda_{r-s}\Rd$ of $\alpha$ by $\beta$ as the alternating $(r-s)$-tensor such that $(\alpha \Contraction \beta)\cdot\gamma = \alpha \cdot (\gamma \wedge\beta)$ for each $\gamma\in\Lambda_{r-s}\Rd$.
	\end{defi}
	
	The following are properties of the contraction.
	
	\begin{lema}\label{lemma: properties of contraction}
		Let $\{v_{1},\ldots,v_{d}\}$ be an orthonormal basis of $\Rd$.
		\begin{enumerate}[label=(\roman*)]
			\item\label{lemma: multilinear algebra i} 
			If $\alpha, \beta \in \Lambda_r \Rd$ for some $r \in \N$ then $\alpha \Contraction \beta = \alpha \cdot \beta$.

			\item\label{lemma: multilinear algebra ii} Let $1\leq s \leq r$, let $1\leq i_{1}<\cdots<i_{r}\leq d$ and $1\leq j_{1}<\cdots<j_{s}\leq d$. If $\{ i_{1}, \ldots, i_{r} \} \cap \{ j_{1}, \ldots, j_{s} \} = \emptyset$ then
			\begin{equation*}
				v_{i_{1}} \wedge \cdots \wedge v_{i_{r}} \Contraction v_{j_{1}}\wedge \cdots \wedge v_{j_{s}} = 0.
			\end{equation*}
			
			\item\label{lemma: multilinear algebra iii} Consider $n = v_{d}$. If $1\leq i_{1} < \cdots < i_{k+1}\leq d-1$, then
			\begin{equation}\label{eq: characterization of contraction when n not a component}
				v_{i_{1}}\wedge \cdots \wedge v_{i_{k+1}} \Contraction n
				= 
				0 .
			\end{equation}
			If $1\leq i_{1} < \cdots < i_{k} \leq d-1$, then
			\begin{equation}\label{eq: characterization of contraction when n yes a component}
				v_{i_{1}} \wedge \cdots \wedge v_{i_{k}} \wedge n\Contraction n = v_{i_{1}} \wedge \cdots \wedge v_{i_{k}}.
			\end{equation}
			
		\end{enumerate}
	\end{lema}
	\begin{proof}
		
		\ref{lemma: multilinear algebra i}
		The contraction $\alpha\Contraction\beta$ is the only constant such that $( \alpha \Contraction \beta ) \gamma = \alpha \cdot ( \gamma \wedge \beta ) = \gamma \alpha \cdot \beta$ for each $\gamma\in\Lambda_{0}\Rd \simeq \R$.
		
		\ref{lemma: multilinear algebra ii} Let $1\leq l_{1}<\cdots<l_{r-s} \leq d$.
		Then
		\[
		( v_{i_{1}} \wedge \cdots \wedge v_{i_{r}} \Contraction v_{j_{1}}\wedge \cdots \wedge v_{j_{s}} )\cdot ( v_{l_{1}}\wedge \cdots \wedge v_{l_{r-s}} )
		= (v_{i_{1}}\wedge \cdots \wedge v_{i_{r}}) \cdot (v_{l_{1}}\wedge \cdots \wedge v_{l_{r-s}} \wedge v_{j_{1}}\wedge \cdots \wedge v_{j_{s}}) = 0 ,
		\]
		where the latter equality is due to Lemma \ref{lemma inner product on basis}\ref{lemma inner product on basis i}, since $\{ i_{1}, \ldots, i_{r} \} \cap \{ j_{1}, \ldots, j_{s} \} = \emptyset$.
		
		\ref{lemma: multilinear algebra iii} Equation \eqref{eq: characterization of contraction when n not a component} is a direct consequence of \ref{lemma: multilinear algebra i}. As for \eqref{eq: characterization of contraction when n yes a component}, let $1\leq l_{1}<\cdots<l_{k} \leq d$ and compute
		\[
		(v_{i_{1}} \wedge \cdots \wedge v_{i_{k}} \wedge n\Contraction n) \cdot ( v_{l_{1}} \wedge \cdots \wedge v_{l_{k}} ) = ( v_{i_{1}} \wedge \cdots \wedge v_{i_{k}} \wedge n ) \cdot (v_{l_{1}} \wedge \cdots \wedge v_{l_{k}} \wedge n) .
		\]
		By Lemma \ref{lemma inner product on basis},
		\[
		( v_{i_{1}} \wedge \cdots \wedge v_{i_{k}} \wedge n ) \cdot (v_{l_{1}} \wedge \cdots \wedge v_{l_{k}} \wedge n) = 0 = ( v_{i_{1}} \wedge \cdots \wedge v_{i_{k}} ) \cdot (v_{l_{1}} \wedge \cdots \wedge v_{l_{k}} )\quad \text{if } \{i_{1} , \ldots , i_{k} \} \neq \{ l_{1} , \ldots , l_{k} \}
		\]
		and
		\[
		( v_{i_{1}} \wedge \cdots \wedge v_{i_{k}} \wedge n ) \cdot (v_{l_{1}} \wedge \cdots \wedge v_{l_{k}} \wedge n) = \sign \sigma =  ( v_{i_{1}} \wedge \cdots \wedge v_{i_{k}} ) \cdot (v_{l_{1}} \wedge \cdots \wedge v_{l_{k}} ) \quad \text{if } \{i_{1} , \ldots , i_{k} \} = \{ l_{1} , \ldots , l_{k} \} ,
		\]
		where $\sigma$ is the only permutation such that $i_{\sigma (j)} = l_j$ for all $1 \leq j \leq k$.
	\end{proof}
	
	The following type of maps are of particular importance in the development of \cite{Silhavy}.
	
	\begin{defi}\label{def: linear map with wedge product}
		Let $k \in \N$, let $A \in \mathcal{L} (\Lambda_{k} \Rd ; \Lambda_{k} \Rd )$ and let $\beta \in \Rd$. 
		We define the map $A \wedge \beta \in \mathcal{L} ( \Lambda_{k+1} \Rd ; \Lambda_{k} \Rd )$ by $( A \wedge \beta) \alpha \coloneqq A ( \alpha \Contraction \beta)$ for each  $\alpha\in\Lambda_{k+1}\Rd$.
	\end{defi}
	
	The following are properties of the map defined above.
	Recall from Section \ref{subsec: minors fixed x} the notation of the minors.
	
	\begin{lema}\label{lemma: properties of maps with wedge product}
		Let $k \in \N$, let $A \in \mathcal{L} (\Lambda_k \Rd ; \Lambda_k \Rd )$ and let $F\in\mathcal{L}(\Rd;\Rd)$. Let $\mathcal{V} = \{v_{1},\ldots,v_{d}\}$ be an orthonormal basis of $\Rd$ and consider $n = v_{d}$.
		
		\begin{enumerate}[label=(\roman*)]
			
			\item\label{properties of map with wedge iii} The map $A\wedge n\in\mathcal{L}(\Lambda_{k+1}\Rd;\Lambda_{k}\Rd)$ is characterized as follows: for $1\leq j_{1}<\cdots<j_{k+1}\leq d$,
			\begin{equation*}
				(A \wedge n)v_{j_{1}} \wedge \cdots \wedge v_{j_{k+1}}
				=
				\begin{cases}
					0 & \text{if}\ j_{k+1}<d,\\
					A(v_{j_{1}} \wedge \cdots \wedge v_{j_{k}}) & \text{if}\ j_{k+1}=d.
				\end{cases}
			\end{equation*}

			\item\label{properties of map with wedge i} If $\beta \in \Rd$ then $\Lambda_0 F \wedge\beta = \beta$.

			\item\label{properties of map with wedge minors} If $1\leq j_{1} < \cdots < j_{k} \leq d$ then 
			\begin{equation*}
				\Lambda_{k}F (v_{j_{1}} \wedge \cdots \wedge v_{j_{k}}) 
				=
				\sum_{ 1\leq i_{1} < \cdots < i_{k}\leq d } M_{ \substack{i_{1}\ldots i_{k} \\ j_{1}\ldots j_{k}} }(F) v_{i_{1}} \wedge \cdots \wedge v_{i_{k}} ,
			\end{equation*}
			where the minors $M_{ \substack{i_{1}\ldots i_{k} \\ j_{1}\ldots j_{k}} }(F)$ are taken with respect to the basis $\mathcal{V}$.

			\item\label{properties of map with wedge v} If $1\leq j_{1} < \cdots < j_{k} \leq d-1$ then 
			\begin{equation*}
				(\Lambda_{k}F\wedge n) v_{j_{1}} \wedge \cdots \wedge v_{j_{k}} \wedge n 
				=
				\sum_{ 1\leq i_{1} < \cdots < i_{k}\leq d } M_{ \substack{i_{1}\ldots i_{k} \\ j_{1}\ldots j_{k}} }(F) v_{i_{1}} \wedge \cdots \wedge v_{i_{k}} .
			\end{equation*}
			
		\end{enumerate}

	\end{lema}
	\begin{proof}
		\ref{properties of map with wedge iii} By Lemma \ref{lemma: properties of contraction}\ref{lemma: multilinear algebra iii}, if $j_{k+1}< d$,
		\begin{equation*}
			(A\wedge n)v_{j_{1}} \wedge \cdots \wedge v_{j_{k+1}}
			=
			A (v_{j_{1}} \wedge \cdots \wedge v_{j_{k+1}} \Contraction n)
			=
			0
		\end{equation*}
		while
		\begin{equation*}
			\begin{split}
				(A\wedge n) v_{j_{1}} \wedge \cdots \wedge v_{j_{k}} \wedge n
				=
				A (v_{j_{1}} \wedge \cdots \wedge v_{j_{k}} \wedge n\Contraction n)
				=
				A (v_{j_{1}} \wedge \cdots \wedge v_{j_{k}} ).
			\end{split}
		\end{equation*}
		
		\ref{properties of map with wedge i} Let $\alpha\in\Lambda_{1}\Rd \simeq \Rd$. By Lemma \ref{lemma: properties of contraction}\ref{lemma: multilinear algebra i} we have that $(\Lambda_0 F\wedge\beta)\alpha = \Lambda_0 F ( \alpha \Contraction \beta) = \alpha \Contraction \beta = \alpha\cdot\beta$.
		
		\ref{properties of map with wedge minors} 
		Let $f_{ij} = F v_j \cdot v_i$ for each $1 \leq i, j \leq d$.
		We compute 	
		\begin{equation*}
			\begin{split}
				\Lambda_{k}F (v_{j_{1}} \wedge \cdots \wedge v_{j_{k}}) &= Fv_{j_{1}} \wedge \cdots \wedge Fv_{j_{k}} \\
				&= \sum_{i_{1}=1}^{d}f_{i_{1}j_{1}}v_{i_{1}} \wedge \cdots \wedge \sum_{i_{k}=1}^{d}f_{i_{k}j_{k}}v_{i_{k}} \\
				& = \sum_{1\leq i_{1},\ldots, i_{k}\leq d} f_{i_{1}j_{1}} \cdots f_{i_{k}j_{k}} v_{i_{1}} \wedge \cdots \wedge v_{i_{k}} \\
				& = \sum_{1\leq i_{1}<\cdots< i_{k}\leq d} M_{\substack{ i_{1},\ldots,i_{k} \\ j_{1},\ldots,j_{k} }}(F) v_{i_{1}} \wedge \cdots \wedge v_{i_{k}}.
			\end{split}
		\end{equation*}
		
		\ref{properties of map with wedge v} We have that 
		\begin{align*}
			(\Lambda_{k} F \wedge n) v_{j_{1}}\wedge \cdots \wedge v_{j_{k}} \wedge n = \Lambda_{k} F (v_{j_{1}}\wedge \cdots \wedge v_{j_{k}}) & = \sum_{1\leq i_{1}<\cdots< i_{k}\leq d} M_{\substack{ i_{1},\ldots,i_{k} \\ j_{1},\ldots,j_{k} }}(F) v_{i_{1}} \wedge \cdots \wedge v_{i_{k}} 
		\end{align*}
		by \ref{properties of map with wedge iii} and \ref{properties of map with wedge minors}.
	\end{proof}

	As seen in Lemma \ref{lemma: properties of maps with wedge product}, the coefficients of $\Lambda_{k}F \wedge n$ with respect to $\mathcal{V}_{k+1}$ and $\mathcal{V}_{k}$ are either zero or the minors of $F$ involving $\{v_{1},\ldots,v_{d-1}\}$.
	
	We now relate the tangential polyconvexity from Definition \ref{defi tangential polyconvexity} with the maps introduced in Definition \ref{def: linear map with wedge product}.
	In the rest of the section we will use the set $\mathcal{G} \coloneqq \{(F,n)\in\mathcal{L}(\Rd;\R^{m}) \times S^{d-1} : Fn=0 \}$. Besides, returning to the notation of the minors from Section \ref{subsec: minors fixed x}, when the chosen bases have a dependence on some $x$ we will stress this dependence denoting by $M_{x}$, $M_{k,x}$, $M_{x}^{0}$, $M_{x}^{1}$, $M_{k, x}^{0}$ and $M_{k, x}^{1}$ the sequences of minors $M$, $M_{k}$, $M^{0}$, $M^{1}$, $M_{k}^{0}$ and $M_{k}^{1}$ in such bases, respectively.
	
	\begin{prop}\label{prop: interface implies tangentially}
		Let $\Omega\subset\Rd$ be a Lipschitz domain, let $n : \partial \Omega \rightarrow \Rd$ be a measurable map such that $n(x)$ is a unit orthogonal vector to $T_{x}\partial\Omega$ for $\mathcal{H}^{d-1}$-a.e.\ $x \in \partial \Omega$.
		Let $\hat{f}:\mathcal{G}\rightarrow\R\cup\{\infty\}$ be such that there exists a convex map $\Psi: \prod_{k=0}^{d-1}\mathcal{L}( \Lambda_{k+1}\Rd;\Lambda_{k}\R^{m} ) \rightarrow \R\cup\{\infty\}$
		with
		\begin{equation*}
			\hat{f}(F,n)
			=
			\Psi(\Lambda_{0}F\wedge n,\ldots, \Lambda_{d-1}F\wedge n) , \qquad (F,n)\in \mathcal{G} .
		\end{equation*}
		Define $W_{0}:T^{d}\partial\Omega \rightarrow \R$ as $W_{0}(x,F) \coloneqq \hat{f}(F_x, n(x))$ where $F_x$ is the linear extension of $F$ to $\Rd$ such that $F_x n(x)=0$ for $\mathcal{H}^{d-1}$-a.e.\ $x\in\partial\Omega$. Then $W_{0}$ is tangentially polyconvex.
	\end{prop}
	\begin{proof}
		Let $\{ v_{1},\ldots,v_{d-1} \}$ be an orthonormal measurable basis of $T\partial\Omega$ and consider $v_{d}=n$, then $\mathcal{V} (x) := \{v_{1} (x),\ldots, v_{d} (x)\}$ is an orthonormal basis of $\Rd$, for $\mathcal{H}^{d-1}$-a.e.\ $x\in\partial\Omega$.
		Let $1 \leq k \leq d-1$ and define $\theta_{k} = \binom{d}{k}^2$.
		We number the elements of $\R^{\theta_k}$ as $( a_{j_{1},\ldots,j_{k}}^{i_{1},\ldots,i_k} )_{1 \leq j_{1} < \cdots < j_{k} \leq d}^{1 \leq i_{1} < \cdots < i_k \leq d}$.
		Define the linear map $\mathfrak{F}_{k,x} : \R^{\theta_k} \rightarrow \mathcal{L}(\Lambda_{k}\Rd ; \Lambda_{k}\Rd)$ as follows: $\mathfrak{F}_{k,x} \left( ( a_{j_{1},\ldots,j_{k}}^{i_{1},\ldots,i_k} )_{1 \leq j_{1} < \cdots < j_{k} \leq d}^{1 \leq i_{1} < \cdots < i_k \leq d} \right)$ is the only linear map such that for each $1 \leq j_{1} < \cdots < j_k \leq d$,
		\[
		\mathfrak{F}_{k,x} \left( ( a_{j_{1},\ldots,j_{k}}^{i_{1},\ldots,i_k} )_{1 \leq j_{1} < \cdots < j_{k} \leq d}^{1 \leq i_{1} < \cdots < i_k \leq d} \right) (v_{j_{1}} (x) \wedge \cdots \wedge v_{j_k} (x))
		=
		\sum_{ 1\leq i_{1} <\cdots< i_{k}\leq d } a_{j_{1},\ldots,j_{k}}^{i_{1},\ldots,i_k} v_{i_{1}} (x) \wedge \cdots \wedge v_{i_{k}} (x).
		\]
		
		Recalling the order of the minors established in Section \ref{subsec: minors fixed x} and thanks to Lemma \ref{lemma: properties of maps with wedge product}\ref{properties of map with wedge minors} we have that $\mathfrak{F}_{k,x} (M_{k,x}(F_x)) = \Lambda_k F_x$. Now define 
		\begin{align*}
			\Phi : \partial\Omega \times \R^{\nu_{d-1}} & \rightarrow \R\\
			(x , (a_{1}, \ldots, a_{ \nu_{d-1} }) ) & \mapsto \Psi( n(x) , \mathfrak{F}_{1,x}( a_{1}, \ldots, a_{\nu_{1}} )\wedge n(x), \ldots, \mathfrak{F}_{d-1,x}( a_{\nu_{d-2}+1}, \ldots, a_{\nu_{d-1}} )\wedge n(x) ) .
		\end{align*}
		Then, $\Phi (x, \cdot)$ is convex for $\mathcal{H}^{d-1}$-a.e.\ $x\in\partial\Omega$, as a composition of a linear with a convex map.
		Moreover, for $(x, F) \in T^d \partial \Omega$, 
		\[
		W_{0}(x,F)
		=
		\hat{f}(F_x ,n(x))
		=
		\Psi(\Lambda_{0}F_x \wedge n(x),\ldots, \Lambda_{d-1}F_x \wedge n(x))  , \qquad (x, F) \in T^d \partial \Omega .
		\]
		Now,
		\[
		M_{x}(F) = \left( M^{0}_{x} (F), M_{x}^{1} (F) \right) =  \left( M_{x}^{0} (F_x), M_{x}^{1} (F) \right) =  \left( M_{1,x} (F_x), \ldots, M_{d-1,x} (F_x), M_{x}^{1} (F)  \right) ,
		\]
		so, recalling Lemma \ref{lemma: properties of maps with wedge product}\ref{properties of map with wedge i} we have that
		\begin{align*}
			\Phi(x, M_{x}(F)) & = \Phi\left( x, M_{1,x} (F_x), \ldots, M_{d-1,x} (F_x), M_{x}^{1} (F) \right) \\
			& = \Psi( n(x) , \mathfrak{F}_{1,x}( M_{1,x} (F_x) ) \wedge n(x), \ldots, \mathfrak{F}_{d-1,x}( M_{d-1,x} (F_x) ) \wedge n(x) ) \\
			& =  \Psi (n(x), \Lambda_{1} F_x \wedge n(x),\ldots, \Lambda_{d-1}F_x \wedge n(x))\\
			&= \Psi(\Lambda_{0}F_x \wedge n(x), \Lambda_{1} F_x \wedge n(x), \ldots, \Lambda_{d-1}F_x \wedge n(x)) ,
		\end{align*}
		which proves the result.
	\end{proof}
	
	Note that condition $Fn=0$ in the definition of $\mathcal{G}$ does not play an essential role: in the proof above we pass from a linear map defined in $T_x \partial \Omega$ to a linear extension to $\Rd$, and $Fn=0$ just fixes a specific extension. A partial converse to the above result also holds.
	
	\begin{prop}\label{prop tangentially implies interface}
		Let $\Omega\subset\Rd$ be of class $C^1$, let $n : \partial \Omega \rightarrow S^{d-1}$ be the unit outward normal to $\Omega$.
		Then there exists a measurable map $S^{d-1} \ni m \mapsto x_m \in \partial\Omega$ such that $n (x_m) = m$ for all $m \in S^{d-1}$.
		Now, let $U:T^{d}\partial\Omega\rightarrow \R$ be tangentially polyconvex.
		Define $\hat{f} : \mathcal{G} \rightarrow \R \cup \{\infty\}$ as 
		$
		\hat{f}( F , m)
		\coloneqq
		U(x_m , F_{| T_{x_m} \partial\Omega} ) .
		$ 
		Then there exists a measurable map $\Psi: S^{d-1} \times \prod_{k=0}^{d-1}\mathcal{L}( \Lambda_{k+1}\Rd;\Lambda_{k}\R^{m} ) \rightarrow \R\cup\{\infty\}$ such that $\Psi (m , \cdot)$ is convex for all $m \in S^{d-1}$ and 
		\begin{equation*}
			\hat{f}(F,n)
			=
			\Psi(n, \Lambda_{0}F\wedge n,\ldots, \Lambda_{d-1}F\wedge n) , \qquad (F,n)\in \mathcal{G} .
		\end{equation*}
	\end{prop}
	\begin{proof}
		The normal $n$ is in fact the Gauss map of $\partial \Omega$, which is known to be surjective (see, e.g., \cite[Chapter 6]{JAThorpe}).
		Now let $F: S^{d-1} \rightarrow \mathcal{P} (\partial\Omega)$ be the set-valued map defined by $F(m) \coloneqq n^{-1}(m)$.
		As $n$ is continuous, $F(m)$ is closed.
		Moreover, as $n$ is surjective, $F(m)$ is non-empty.
		Now we show that $F$ is measurable in the sense of \cite[Definition 8.1.1]{AubinFrankowska}: for each relatively open subset $\mathcal{O}\subseteq\partial\Omega$, the set $\{ m\in S^{d-1} : F(m) \cap \mathcal{O} \neq \emptyset \}$ must be Borel.
		To check this, we express
		\[
		\{ m\in S^{d-1} : F(m) \cap \mathcal{O} \neq \emptyset \} = n (\mathcal{O}) 
		\]
		and $\mathcal{O}$ as a countable union of compact sets: $\mathcal{O} = \bigcup_{m=1}^{\infty} K_{m}$.
		Since $n$ is continuous, $n (K_m)$ is compact for each $m \in \N$, so $n(\mathcal{O})$ is Borel as a countable union of compact sets.
		An application of \cite[Theorem 8.1.3]{AubinFrankowska} concludes that there exists a measurable map $S^{d-1} \ni m \mapsto x_m \in \partial\Omega$ such that $n (x_m) = m$.
		
		There exist 
		$
		\mathcal{V} = \{ v_{1}, \ldots, v_{d-1} \} 
		$
		a measurable basis of $T\partial\Omega$ and map $\Phi : \partial\Omega \times \R^{\nu_{d-1}} \rightarrow \R$ such that $\Phi(x,\cdot)$ is convex for $\mathcal{H}^{d-1}$-a.e.\ $x\in\partial\Omega$ and 
		$
		U(x,F ) 
		=
		\Phi(x , M(F))
		$
		for $\mathcal{H}^{d-1}$-a.e.\ $( x,F ) \in T^{d}\partial\Omega$, where $M(F)$ is taken with respect to the basis $\{ v_{1}(x), \ldots, v_{d-1}(x) \} $ and the canonical basis in $\Rd$. 
		Let $v_{d}=n$ and consider $\tilde{\mathcal{V}} = \{v_{1},\ldots,v_{d-1},v_{d}\}$ as a measurable orthonormal basis of $\Rd$.
		Let $k \leq d-1$ and $\tilde{\theta}_k = \binom{d-1}{k}\binom{d}{k}$; we number the elements of $\R^{\tilde{\theta}_k}$ as $( a_{j_{1},\ldots,j_{k}}^{i_{1},\ldots,i_{k}} )_{1 \leq j_{1} < \cdots < j_{k} \leq d}^{1 \leq i_{1} < \cdots < i_{k} \leq d-1}$.
		Define the linear map 
		\begin{align*}
			\mathfrak{C}_{k,x} : \mathcal{L}( \Lambda_{k+1}\Rd ; \Lambda_{k}\Rd ) & \rightarrow \R^{\tilde{\theta}_k}\\
			A & \mapsto \left( A (v_{j_{1}}(x) \wedge \cdots \wedge v_{j_{k}}(x) \wedge n(x)) \cdot ( v_{i_{1}}(x) \wedge \cdots \wedge v_{i_{k}}(x) ) \right)_{ 1 \leq j_{1}<\cdots<j_{k} \leq d-1}^{1 \leq i_{1}<\cdots < i_{k} \leq d } .
		\end{align*}
		Note that map $\mathfrak{F}_{k,x}$ of Proposition \ref{prop: interface implies tangentially} is an isomorphism with inverse
		\begin{align*}
			\mathfrak{F}_{k,x}^{-1} : \mathcal{L}( \Lambda_{k}\Rd ; \Lambda_{k}\Rd ) & \rightarrow \R^{ \theta_k }\\
			A & \mapsto \left(  A ( v_{j_{1}} (x) \wedge \cdots \wedge v_{j_{k}} (x)) \cdot ( v_{i_{1}} (x) \wedge \cdots \wedge v_{i_{k}} (x) ) \right)_{1 \leq j_{1}<\cdots<j_{k} \leq d}^{1 \leq i_{1}<\cdots<i_{k} \leq d}.
		\end{align*}
		As a consequence of Lemma \ref{lemma: properties of maps with wedge product}\ref{properties of map with wedge iii}, if $B\in\mathcal{L}(\Lambda_{k}\Rd ; \Lambda_{k}\Rd)$, then
		\[
		\mathfrak{C}_{k,x}( B\wedge n ) = \left( \mathfrak{F}_{k,x} ^{-1}( B ) \right)_{ 1 \leq j_{1}<\cdots<j_{k} \leq d-1}^{1 \leq i_{1}<\cdots < i_{k} \leq d } .
		\]
		Using Lemma \ref{lemma: properties of maps with wedge product}\ref{properties of map with wedge minors}, if $F \in \mathcal{L} (\Rd; \Rd)$ then $\mathfrak{F}_{k,x} ^{-1}( \Lambda_k F ) = M_{k,x} (F)$ with respect to $\tilde{\mathcal{V}}_k (x)$.
		By Lemma \ref{lemma: properties of maps with wedge product}\ref{properties of map with wedge v}, $\mathfrak{C}_{k,x}( \Lambda_k F\wedge n ) = M^0_{k,x} (F)$.
		Finally, by \eqref{eq:M0kMk}, $M^0_{k,x} (F) = M_{k,x} (F_{|T_x \partial \Omega})$.
		Altogether,
		\begin{equation}\label{eq:CkxLF}
			\mathfrak{C}_{k,x}( \Lambda_k F\wedge n ) = M_{k,x} (F_{|T_x \partial \Omega}) .
		\end{equation}

		Now for each $x \in \partial \Omega$, define 
		\begin{align*}
			\Psi_x : \prod_{k=0}^{d-1} \mathcal{L}(\Lambda_{k+1}\Rd;\Lambda_{k}\Rd)  & \rightarrow \R\\
			(A_{0}, \ldots, A_{d-1}) & \mapsto \Phi(x, ( \mathfrak{C}_{1,x}(A_{1}), \ldots , \mathfrak{C}_{d-1,x}(A_{d-1}) )),
		\end{align*}
		which is convex, as a composition of a convex and a linear map.
		Thanks to \eqref{eq:CkxLF}, for all $F \in \mathcal{L} (\Rd; \Rd)$,
		\[
		\Psi_x (\Lambda_0 F \wedge n(x), \ldots , \Lambda_{d-1}F \wedge n(x) ) = \Phi (x, (M_{1, x} (F_{|T_x \partial \Omega}), \ldots , M_{d-1, x} (F_{|T_x \partial \Omega}) )) = \Phi (x, M (F_{|T_{x} \partial \Omega})) .
		\]
		Then, for all $(F, m) \in \mathcal{G}$,
		\[
		\hat{f} (F, m) 
		=U (x_m, F_{|T_{x_m} \partial \Omega}) 
		=
		\Phi (x_m, M (F_{|T_{x_m} \partial \Omega}))
		=
		\Psi_{x_m} (\Lambda_0 F \wedge m, \ldots , \Lambda_{d-1} F \wedge m ) .
		\]
		The proof is concluded by defining $\Psi (m, A_0, \ldots, A_{d-1}) \coloneqq \Psi_{x_m} (A_0, \ldots, A_{d-1})$.
	\end{proof}
	
	We remarked after Proposition \ref{prop: interface implies tangentially} that condition $Fn=0$ in the definition of $\mathcal{G}$ is not essential.
	The following is a precise statement of this fact.
	
	\begin{prop}\label{prop: equiv between 3 definitions of interface polyconvexity}
		Let $d,m\in\N$, let $t = \min\{d-1, m\}$ and $\hat{f} : \mathcal{G} \rightarrow \R\cup\{\infty\}$.
		The following statements are equivalent:
		\begin{enumerate}[label=(\roman*)]
			\item\label{i: equivalence proposition def interf poly 1} There exists a convex map $\Psi : \prod_{k=0}^{t} \mathcal{L}(\Lambda_{k+1}\Rd;\Lambda_{k}\R^{m}) \rightarrow \R\cup\{\infty \}$ such that $\hat{f}(F,n) = \Psi(\Lambda_{0}F\wedge n, \Lambda_{1}F \wedge n,\ldots, \Lambda_{t}F\wedge n)$ for each $(F,n)\in\mathcal{G}$.
			
			\item\label{iii: equivalence proposition def interf poly 3} There exist $f : \mathcal{L}( \Rd ; \R^{m} ) \times S^{d-1} \rightarrow \R \cup \{\infty\}$ an extension of $\hat{f}$ and $\Psi : \prod_{k=0}^{t} \mathcal{L}(\Lambda_{k+1}\Rd;\Lambda_{k}\R^{m}) \rightarrow \R\cup\{\infty \}$ convex, such that $f(F,n) = \Psi(\Lambda_{0}F\wedge n, \Lambda_{1}F \wedge n,\ldots, \Lambda_{t}F\wedge n)$ for each $(F,n)\in \mathcal{L}(\Rd;\R^{m})\times S^{d-1}$.
		\end{enumerate}
	\end{prop}
	\begin{proof}
		Since $\mathcal{G} \subset \mathcal{L}( \Rd ; \R^{m} ) \times S^{d-1}$, statement \ref{i: equivalence proposition def interf poly 1} is a straightforward consequence of \ref{iii: equivalence proposition def interf poly 3}. The converse implication is also trivial definining the extension $f$ by $f(F,n) \coloneqq \Psi(\Lambda_{0}F\wedge n, \Lambda_{1}F \wedge n,\ldots, \Lambda_{t}F\wedge n)$ for each $(F,n)\in \mathcal{L}( \Rd ; \R^{m} ) \times S^{d-1}$.
	\end{proof}
	
	A definition of interface polyconvexity can be given as follows (see \cite[Definition 5.1, Theorem 5.3]{Silhavy}).
	
	\begin{defi}\label{def: interface polyconvexity 1}
		Let $d,m\in\N$, let $t=\min\{d-1,m\}$ and let $\mathcal{G} \coloneqq \{(F,n)\in\mathcal{L}(\Rd;\R^{m}) \times S^{d-1} : Fn=0 \}$. A map $\hat{f} : \mathcal{G} \rightarrow\R\cup\{\infty\}$ is said to be interface polyconvex if there exists a positively $1$-homogeneous convex map $\Psi: Y\rightarrow \R\cup\{\infty\}$ on
		$
		Y \coloneqq \prod_{k=0}^{t}\mathcal{L}( \Lambda_{k+1}\Rd;\Lambda_{k}\R^{m} )
		$ 
		such that
		\begin{equation*}\label{eq: def 1 interface polyconvexity}
			\hat{f}(F,n) = \Psi(\Lambda_{0}F\wedge n, \Lambda_{1}F\wedge n,\ldots, \Lambda_{t}F\wedge n)
		\end{equation*}
		for each $(F,n)\in \mathcal{G}$.
	\end{defi}
	
	As seen in Proposition \ref{prop: interface implies tangentially} and \ref{prop tangentially implies interface}, the key difference between tangential polyconvexity and interface polyconvexity is that, in the latter, the map $\Psi$ needs to be positively $1$-homogeneous.
	Definition \ref{def: interface polyconvexity 1} comes from a characterization (see \cite[Theorem 5.3]{Silhavy}) more suitable for our framework. The original \cite[Definition 5.1]{Silhavy} defines a map as interface polyconvex if it is a supremum of a family of null Lagrangians, which, by \cite[Theorem 5.2]{Silhavy}, are linear combinations of maps of the form $\Lambda_{k}F\wedge n$ with $0\leq k\leq d-1$. Because of this, such suprema (and hence, the maps $\Psi$ of interfacial polyconvexity) are convex and positively $1$-homogeneous. In contrast, in the case of tangential polyconvexity, the map $\Psi$ only needs to be convex, so it can be expressed as a supremum of a family of affine maps, a property that does not grant the positive $1$-homogeneity. In this sense, interface polyconvexity is a more restrictive concept than tangential polyconvexity.
	
	Positive $1$-homogeneity is not necessary to achieve the lower semicontinuity of the energy functionals, as shown in \cite[Section 6]{Silhavy} and in Section \ref{se:existence} below.
	In \cite{Silhavy}, positive $1$-homogeneity is related to the increase of the area of the so-called competitor interface (in our case, $\partial\Omega$).
	
	\section{Tangential polyconvexity in surface potentials}\label{se:examples}
	
	Explicit examples of the elastic energy $U : \mathcal{D}_{U} \rightarrow \R$ from \eqref{functional to minimize}, related to pressure loading and membrane loading on $\partial\Omega$, can be found in \cite{PodioVergara}. The context of \cite{PodioVergara} requires maps $u \in C^{1}(\overline{\Omega};\Rd)$ and an important role is played by $\cof D u(x) n(x)$ for $x \in \partial \Omega$.
	In our case we work with maps $u \in W^{1, p} (\partial \Omega; \Rd)$, so we have to give a proper definition of $\cof D u(x) n(x)$.
	
	We first state some facts from multilinear algebra complementing those of Section \ref{sec: Interface polyconvexity}.
	Assume that $V\subset\Rd$ is a $(d-1)$-dimensional vector subspace of $\Rd$.
	For $k \in \N$, the space $\Lambda_{k}V$ consists of all alternating $k$-tensors on $V$.
	We will make the natural identifications $\Lambda_{0}V \simeq \Lambda_{d-1} V \simeq \R$ and $\Lambda_{1}V \simeq V$.
	Let $L \in \mathcal{L} (V ; \Rd)$.
	The map $\Lambda_{k}L : \Lambda_{k}V \rightarrow \Lambda_{k}\Rd$ is defined as the only linear map such that $(\Lambda_{k}L)(a_{1}\wedge \cdots \wedge a_{k}) = La_{1} \wedge \cdots \wedge La_{k}$ for $a_{1},\ldots,a_{k}\in V$.

	Let $\{ v_1, \ldots, v_{d-1} \}$ be an orthonormal basis of $V$, let $n$ be a unit normal vector to $V$ and consider $v_{d} = n$.
	The space $\Lambda_{d-1} V$ is generated by $\{v_{1} \wedge \cdots \wedge v_{d-1}\}$ and can be identified with the subspace of $\Rd$ generated by $n$.
	Let $F \in \mathcal{L}(V;\Rd)$. For any $\widetilde{F} \in \mathcal{L}(\Rd ; \Rd)$ extending $F$, the vector $\cof(\widetilde{F})n$ does not depend on the extension $\widetilde{F}$, since the map $\Lambda_{d-1}F$ is determined by the value $\Lambda_{d-1}F(n)$ and formula
	\begin{equation}\label{eq:Lambdacof}
		\Lambda_{d-1}F(n) = (\cof \widetilde{F})n
	\end{equation}
	holds.
	As in Lemma \ref{lemma: properties of maps with wedge product}\ref{properties of map with wedge minors}, the value of  the map $\Lambda_{d-1}F$ can be rewritten in terms of the minors as
	\begin{equation}\label{eq:Lambdad-1}
		\Lambda_{d-1}F (v_{1}\wedge\cdots\wedge v_{d-1})
		=
		\sum_{1\leq i_{1} < \cdots < i_{d-1} \leq d} M_{\substack{ i_{1} ,\ldots, i_{d-1} \\ 1, \ldots, d-1 }}(F) v_{i_1} \wedge \cdots \wedge v_{i_{d-1}}
		=
		\sum_{i=1}^{d} (-1)^{d-i} M_{d-1}(F)_{i}v_{i},
	\end{equation}
	where the minors are taken with respect to the chosen bases and we have made the identifications $v_1 \wedge \cdots \wedge v_{i-1} \wedge v_{i+1} \wedge \cdots \wedge v_d = (-1)^{d-i} v_i$ for all $i = 1, \ldots, d$.
	
	The results of the previous paragraph are now applied to $V = T_x \partial\Omega$ for varying $x \in \partial \Omega$.
	Let $\{ v_{1},\ldots,v_{d-1} \}$ be an orthonormal measurable basis of $T\partial\Omega$, let $n : \partial \Omega \rightarrow \Rd$ be the unit outward normal to $\Omega$ and consider $v_{d} = n$.
	Fix $x\in\partial\Omega$.
	Given $F \in \mathcal{L}(T_{x}\partial\Omega;\Rd)$ and any extension of it $F_x \in \mathcal{L}(\Rd;\Rd)$, by \eqref{eq:Lambdacof} and \eqref{eq:Lambdad-1},
	\begin{equation}\label{eq:cofFx}
		(\cof F_x) n(x) = \Lambda_{d-1}F (v_{1}(x)\wedge\cdots\wedge v_{d-1}(x))
		=
		\sum_{i=1}^{d} (-1)^{d-i} M_{x,d-1}(F)_{i}v_{i}(x)
	\end{equation}
	and we will apply this formula to $F = D^{\tau}u(x)$.
	
	Some examples in \cite{PodioVergara} of the boundary energy functional from \eqref{functional to minimize} are 
	\[
	\int_{\partial\Omega} U_{i}(x,y,F,n) \dd \mathcal{H}^{d-1}(x), \qquad i=1,2 ,
	\]
	with, in their notation,
	\[
	U_{1}(x,y,F,n) = \pi(y)y\cdot \cof Du(x)n(x) \qquad \text{and} \qquad U_{2}(x,y,F,n) = \varepsilon_{0} \abs{\cof Du(x) n(x)}.
	\] 
	The expression of $U_{1}$ corresponds to a body having pressure interaction with its environment (see \cite[Proposition 5.1]{PodioVergara}) with $\pi : \Rd \rightarrow \R_{+}$ being some pressure function depending on the traction boundary condition. The expression of $U_{2}$ corresponds to a body with membrane interaction with its environment (see \cite[Proposition 5.3]{PodioVergara}); intuitively, this is a body with an elastic membrane glued to it with $\varepsilon_{0}>0$ being a constant representing the material modulus of the membrane.
	
	These examples can be rewritten with our notation as follows. Using \eqref{eq:cofFx}, the case of pressure interaction has the expression
	\begin{equation*}\label{eq: example pressure loading surface energy potential}
		U_{1}( x, y, F , n ) = \pi(y) y \cdot \left( \sum_{i=1}^{d} (-1)^{d-i} M_{x,d-1}(F)_{i}v_{i}(x) \right) ,
	\end{equation*}
	which is linear with respect to the minors of $F$, and thus tangentially polyconvex.
	The case of the membrane interaction has the expression
	\begin{equation*}\label{eq: example membrane loading surface energy potential}
		U_{2}( x, y , F , n ) = \varepsilon_{0} \left( \sum_{i=1}^{d}\left( M_{x,d-1}(F)_{i} \right)^2 \right)^{1/2},
	\end{equation*}
	which is convex with respect to the minors of $F$, and thus tangentially polyconvex.
	
	Suitable examples of energy functions should be coercive (see Theorem \ref{teorema existencia minimizadores} below).
	Neither $U_1$ nor $U_2$ satisfy this condition.
	Nevertheless, if we define the energies as
	\begin{equation*}\label{V1 modificada}
		\overline{U_{i}}( x, y, F, n ) \coloneqq U_{i}( x, y, F , n ) + c |F|^{p}, \qquad i =1,2 ,
	\end{equation*}
	we achieve the coercivity and retain the tangential polyconvexity, provided $c>0$ and $p>1$.
	
	\section{Existence of minimizers}\label{se:existence}
	
	In this section, we prove the existence of minimizers of the functional $I$ in \eqref{functional to minimize} in the class $\overline{\mathcal{A}}_{p}(\Omega)\cap\AIB$ under some natural conditions on the integrands; essentially, polyconvexity of $W$, tangential polyconvexity of $U$ and standard coercivity assumptions. 
	
	The compactness of $\overline{\mathcal{A}}_{p}(\Omega)$ shown in \cite[Prop.\ 10.2]{Henao-Mora-Oliva} together with the compactness of $\AIB$ given by Lemma \ref{compactness of AIB} imply to the following result.
	
	\begin{prop}\label{prop 10.2}
		Let $p>d-1$. Let $\{u_{j}\}_{j\in\N}\subset\overline{\mathcal{A}}_{p}(\Omega)\cap\AIB$ be such that $\{u_{j}\}_{j\in\N}$ is bounded in $W^{1,p}(\Omega;\Rd)\cap W^{1,p}(\partial\Omega;\Rd)$ and $\{\det Du_{j}\}_{j\in\N}$ is equiintegrable.
		Then there exists $u\in\overline{\mathcal{A}}_{p}(\Omega)\cap\AIB$ such that, for a subsequence, $$u_{j}\rightharpoonup u\qquad\text{in\ }W^{1,p}(\Omega;\Rd)\cap W^{1,p}(\partial\Omega;\Rd)\qquad\text{and}\qquad\det Du_{j}\rightharpoonup\det Du\qquad\text{in\ }L^{1}(\Omega)$$as $j\to\infty$.
	\end{prop}
	
	We now state some elementary Poincar\'e inequalities.
	
	\begin{lema}\label{lema 2 Poincaré type ineq}
		Let $p \geq 1$.
		Let $\Omega\subset\Rd$ be a bounded Lipschitz open set such that $\partial\Omega$ is connected, let $\Gamma\subseteq\partial\Omega$ be a rectificable set with positive $\mathcal{H}^{d-1}$ measure.
		Then there exists $C>0$ such that for all $u\in W^{1,p}(\partial\Omega)$ with $u_{|\Gamma}= 0$, one has
		\begin{equation}\label{Poincare type ineq lema 1}
			\|u\|_{L^{p}(\partial\Omega)}\leq C\|D^{\tau}u\|_{L^{p}(\partial\Omega)} .
		\end{equation}
	\end{lema}
	
	\begin{lema}\label{lema 3 Poincaré type ineq}
		Let $p \geq 1$.
		Let $\Omega\subset\Rd$ be a bounded Lipschitz domain.
		Then there exists $C>0$ such that for all $u\in W^{1,p}(\Omega)\cap L^{p}(\partial\Omega)$ with
		\begin{equation}\label{integral condition lemma 3}
			\int_{\partial\Omega}u(x)\dd\mathcal{H}^{d-1}(x)=0
		\end{equation}
		one has
		\begin{equation*}
			\|u\|_{L^{p}(\Omega)}\leq C\|Du\|_{L^{p}(\Omega)}.
		\end{equation*}
	\end{lema}
	
	\begin{lema}\label{lema 4 Poincaré type ineq}
		Let $p \geq 1$.
		Let $\Omega\subset\Rd$ be a bounded Lipschitz open set such that $\partial\Omega$ is connected.
		Then there exists $C>0$ such that for all $u\in W^{1,p}(\partial\Omega)$ with \eqref{integral condition lemma 3}, one has \eqref{Poincare type ineq lema 1}.
	\end{lema}
	
	Lower semicontinuity for tangentially quasiconvex integrands was proved in \cite[Proposition 2.5]{DaFoMaTr99}. 
	We now prove the lower semicontinuity of the boundary integral of the elastic energy in \eqref{functional to minimize} under the assumptions previously stated on the integrand $U$.
	
	\begin{lema}\label{lemma: lower semicontinuity of I}
		Let $p>d-1$. Recall $\mathcal{D}_{U}$ from Remark \ref{remark: dominio de U} and let $U: \mathcal{D}_{U} \rightarrow\R$ be an  $\mathcal{H}^{d-1}_{|\partial\Omega} \times \mathcal{B}^{d} \times \mathcal{B}^{d\times (d-1)} \times \mathcal{B}_{|S^{d-1}}$-measurable map 
		such that $U(x,\cdot,\cdot,\cdot)$ is lower semicontinuous for $\mathcal{H}^{d-1}$-a.e.\ $x\in\partial\Omega$, 
		such that $U(\cdot,y,\cdot,n)$ is tangentially polyconvex for every $y\in\Rd$ and for every $n\in S^{d-1}$ and 
		such that there exists a constant $c>0$ and a map $a\in L^{1}(\partial\Omega)$ with
		\[
		U(x,y,F,n)\geq a(x)+c\abs{F}^{p}
		\]
		for $\mathcal{H}^{d-1}$-a.e.\ $x\in\partial\Omega$, all $y\in\Rd$, all $F \in \mathcal{L}(T_x \partial \Omega;\Rd)$ and all $n \in S^{d-1}$.
		Then for any $\{u_{j}\}_{j\in\N} \subset W^{1,p}(\partial\Omega;\Rd)$ such that $u_{j}\rightharpoonup u$ in $W^{1,p}(\partial\Omega;\Rd)$ for some $u\in W^{1,p}(\partial\Omega;\Rd)$ we have that
		\[
		\int_{\partial\Omega} U(x,u(x),D^{\tau}u(x),n(x)) \dd\mathcal{H}^{d-1}(x) \leq \liminf_{j \to \infty} \int_{\partial\Omega} U(x, u_j(x), D^{\tau} u_j(x), n(x)) \dd\mathcal{H}^{d-1}(x).
		\]
	\end{lema}
	\begin{proof}
		By Proposition \ref{convergencia de menores} we have that $Ml( D^{\tau}u_{n} ) \rightharpoonup Ml( D^{\tau}u )$ in $L^{1}(\partial\Omega)$ as $n\to\infty$. Since $U(\cdot,y,\cdot,n)$ is tangentially polyconvex for each $y\in\Rd$ and each $n\in S^{d-1}$, let $\Phi: \partial\Omega \times \R^{\nu_{d-1}} \rightarrow \R $ be the map such that $\Phi(x,\cdot)$ is convex for $\mathcal{H}^{d-1}$-a.e.\ $x\in\partial\Omega$ and $U(x,y,F,n) = \Phi(x,M(F))$ for each $F \in \mathcal{L}(T_{x}\partial\Omega;\Rd)$, each $y\in\Rd$ and each $n\in S^{d-1}$, where $F$ is taken as the matrix representation with respect to some measurable basis, which can be taken as an $L^{\infty}$ basis thanks to Proposition \ref{prop: indpendence of measurable basis}, and the canonical basis of $\Rd$. Then we have that
		\begin{equation*}
			\begin{split}
				\int_{\partial\Omega} U(x,u(x),D^{\tau}u(x),n(x)) \dd\mathcal{H}^{d-1}(x) &
				=
				\int_{\partial\Omega} \Phi(x,M( D^{\tau}u(x) )) \dd\mathcal{H}^{d-1}(x)\\
				& \leq
				\liminf_{j\to\infty} \int_{\partial\Omega} \Phi(x,M( D^{\tau}u_{j}(x) )) \dd\mathcal{H}^{d-1}(x) \\
				& =
				\liminf_{j\to\infty} \int_{\partial\Omega} U(x,u_{j}(x),D^{\tau}u_{j}(x),n(x)) \dd\mathcal{H}^{d-1}(x)
			\end{split}
		\end{equation*}
		thanks to \cite[Theorem 7.5]{Fonseca-Leoni} and Definition \ref{def: convergencia de menores de una matriz}.
	\end{proof}
	
	We now show the existence of minimizers.
	As before, we consider $D^{\tau}u$ as a map from $\partial\Omega$ to $\R^{d\times(d-1)}$ by fixing a measurable basis.
	
	\begin{teorema}\label{teorema existencia minimizadores}
		Let $p>d-1$ and let $\Omega\subset\Rd$ be a bounded Lipschitz open set such that $\Rd\setminus\partial\Omega$ has exactly two connected components.
		Let $W:\Omega\times\Rd\times\R_{+}^{d\times d}\rightarrow\R$ and $U: \partial\Omega\times\Rd\times\R^{d\times (d-1)}\times S^{d-1} \rightarrow\R$ satisfy the following conditions:
		\begin{enumerate}[label=(\roman*)]
			\item\label{a teorema existencia minimizadores}
			$W$ is $\mathcal{L}^{d}\times\mathcal{B}^{d}\times\mathcal{B}^{d\times d}$-measurable and $U$ is $\mathcal{H}^{d-1}_{|\partial\Omega} \times \mathcal{B}^{d} \times \mathcal{B}^{d\times (d-1)} \times \mathcal{B}_{|S^{d-1}}$-measurable, where $\mathcal{B}^{d}$ denotes the Borel $\sigma$-algebra in $\R^{d}$.
			
			\item\label{b teorema existencia minimizadores}
			$W(x,\cdot,\cdot)$ and $U(x,\cdot,\cdot,\cdot)$ are lower semicontinuous for a.e.\ $x\in\Omega$ and $\mathcal{H}^{d-1}$-a.e.\ $x\in\partial\Omega$, respectively.
			
			\item\label{c teorema existencia minimizadores}
			For a.e.\ $x\in\Omega$ and every $y\in\Rd$, the function $W(x,y,\cdot)$ is polyconvex; and for every $y\in\Rd$ and for every $n\in S^{d-1}$, the function $U(\cdot,y,\cdot,n)$ is tangentially polyconvex.
			
			\item\label{d teorema existencia minimizadores}
			There exist constants $c_{1},c_{2}>0$, functions $a_{1}\in L^{1}(\Omega)$, $a_{2}\in L^{1}(\partial\Omega)$ and a Borel function $h:(0,\infty)\rightarrow[0,\infty)$ such that
			\[
			\lim_{t\searrow0}h(t)=\lim_{t\to\infty}\frac{h(t)}{t}=\infty,
			\]
			\[
			W(x,y,F)\geq a_{1}(x)+c_{1}\abs{F}^{p}+h(\det F)\quad\text{for a.e.\ }x\in\Omega,\ \text{all\ }y\in\Rd\text{ and all\ }F\in\R_{+}^{d\times d}
			\]
			and
			\[
			U(x,y,F,n)\geq a_{2}(x)+c_{2}\abs{F}^{p}\quad\text{for $\mathcal{H}^{d-1}$-a.e.\ }x\in\partial\Omega,\ \text{ all\ }y\in\Rd, \text{ all\ } F \in \mathcal{L}(T_x \partial \Omega;\Rd) \text{ and all } n \in S^{d-1}.
			\]
		\end{enumerate}
		Let $I$ be as in \eqref{functional to minimize}.
		Consider the following admissible classes:
		\begin{enumerate}[label=\arabic*)]
			\item
			Let $\Gamma$ be a rectificable subset of $\partial\Omega$ with positive $\mathcal{H}^{d-1}$ measure, and let $u_{0} : \Gamma \to \R^d$.
			Define $\mathcal{A}_{1}$ as the set of $u\in\overline{\mathcal{A}}_{p}(\Omega)\cap\AIB$ such that $\det Du>0$ a.e.\  and $u_{|\Gamma}=u_{0|\Gamma}$.
			
			\item
			Define $\mathcal{A}_{2}$ as the set of $u\in\overline{\mathcal{A}}_{p}(\Omega)\cap\AIB$ such that $\det Du>0$ a.e.\  and $$\int_{\partial\Omega}u(x)\dd\mathcal{H}^{d-1}(x)=0.$$
			
			\item
			Let $K\subset\Rd$ be compact.
			Define $\mathcal{A}_{3}$ as the set of $u\in\overline{\mathcal{A}}_{p}(\Omega)\cap\AIB$ such that $\det Du>0$ a.e.\  and $u(x)\in K$ for a.e.\ $x\in\Omega$.
		\end{enumerate}
		Fix $i\in\{1,2,3\}$.
		Assume $\mathcal{A}_{i}\neq\emptyset$ and $I$ is not identically infinity in $\mathcal{A}_{i}$.
		Then there exists a minimizer of $I$ in $\mathcal{A}_{i}$, and any element of $\mathcal{A}_{i}$ is injective a.e.
	\end{teorema}
	\begin{proof}
		Fix $i\in\{1,2,3\}$.
		Let $\{u_{j}\}_{j\in\N}$ be a minimizing sequence of $I$ in $\mathcal{A}_{i}$. Assumption \ref{d teorema existencia minimizadores} implies that both $\{Du_{j}\}_{j\in\N}$ and $\{D^{\tau}u_{j}\}_{j\in\N}$ are bounded in $L^{p}(\Omega;\R^{d\times d})$ and $L^{p}(\partial\Omega; \R^{d\times (d-1)})$, respectively.
		
		Let us see that $\{u_{j}\}_{j\in\N}$ is bounded in $W^{1,p}(\Omega;\Rd) \cap W^{1,p}(\partial \Omega;\Rd)$.
		We will use that $\Omega$ and $\partial \Omega$ are connected (Proposition \ref{proposition of the connectedness of boundary}).
		In the set $\mathcal{A}_{1}$, because $u_{j|\Gamma}=u_{0|\Gamma}$ for any $j\in\N$, Poincaré's inequality gives us that $\{u_{j}\}_{j\in\N}$ is bounded in $L^{p}(\Omega;\R^{d})$, while Lemma \ref{lema 2 Poincaré type ineq} gives the boundedness of $\{u_{j}\}_{j\in\N}$ in $L^{p}(\partial\Omega;\Rd)$.
		In the case of $\mathcal{A}_{2}$, Lemma \ref{lema 3 Poincaré type ineq} gives us the boundedness of $\{u_{j}\}_{j\in\N}$ in $L^{p}(\Omega;\Rd)$, while Lemma \ref{lema 4 Poincaré type ineq} gives us the boundedness of $\{u_{j}\}_{j\in\N}$ in $L^{p}(\partial\Omega;\Rd)$.
		For the set $\mathcal{A}_{3}$, as $K$ is compact, $\{u_{j}\}_{j\in\N}$ is bounded in $L^{\infty}(\Omega;\Rd)$ and therefore in $W^{1,p}(\Omega;\Rd)$.
		By continuity of the trace operator, $\{u_{j}\}_{j\in\N}$ is bounded in $L^{p}(\partial\Omega;\Rd)$.
		In the three cases, $\{u_{j}\}_{j\in\N}$ is bounded in $W^{1,p}(\Omega;\Rd) \cap W^{1,p}(\partial \Omega;\Rd)$.
		
		Assumption \ref{d teorema existencia minimizadores} on $h$ and De la Vall\'ee Poussin's criterion imply that $\{\det Du_{j}\}_{j\in\N}$ is equiintegrable.
		By Proposition \ref{prop 10.2}, there exists $u\in\overline{\mathcal{A}}_{p}(\Omega)\cap\AIB$ such that, for a subsequence (not  relabelled),
		\begin{equation}\label{eq:ujtou}
			u_{j}\rightharpoonup u\qquad\text{in\ }W^{1,p}(\Omega;\Rd) \cap W^{1,p}(\partial \Omega;\Rd) \qquad\text{and}\qquad\det Du_{j}\rightharpoonup\det Du\qquad\text{in\ }L^{1}(\Omega)
		\end{equation}
		as $j\to\infty$.
		As $\det Du_{j} > 0$ a.e., we have that $\det Du \geq 0$ a.e.
		Thanks to the assumption on $h$, a standard argument based on Fatou's lemma (see, e.g., \cite[Th.\ 5.1]{Muller-Spector}) shows that $\det Du > 0$ a.e.
		
		As $p > d-1$, a standard result on the continuity of minors (e.g., \cite[Th.\ 8.20]{Dacorogna}) together with \eqref{eq:ujtou} shows that $M (D u_j) \rightharpoonup M (D u)$ in $L^1 (\O, \R^{\nu_{d}})$.
		By the lower semicontinuity of polyconvex functionals (e.g., \cite[Th.\ 5.4]{BaCuOl81} or \cite[Th.\ 7.5]{Fonseca-Leoni}),
		\begin{equation}\label{eq: primera parte semicontinuidad}
			\int_{\Omega} W(x,u(x),Du(x))\dd x \leq \liminf_{j \to \infty} \int_{\Omega} W(x, u_j(x), Du_j(x)) \dd x.
		\end{equation}
		By Lemma \ref{lemma: lower semicontinuity of I} and equation \eqref{eq: primera parte semicontinuidad} we have that $I[u]\leq\liminf_{j\to\infty}I[u_{j}]$.
		
		If $u_j \in \mathcal{A}_{1}$ for all $j \in \N$, then, by continuity of traces, $u_{|\Gamma}=u_{0|\Gamma}$, so $u \in \mathcal{A}_{1}$ and $u$ is a minimizer of $I$ in $\mathcal{A}_{1}$.
		If $u_j \in \mathcal{A}_{2}$ for all $j \in \N$, then $\int_{\partial \Omega} u \dd \mathcal{H}^{d-1} =0$, so $u \in \mathcal{A}_{2}$ and $u$ is a minimizer of $I$ in $\mathcal{A}_{2}$.
		If $u_j \in \mathcal{A}_{3}$ for all $j \in \N$, then, as $K$ is compact, $u (x) \in K$ for a.e.\ $x \in \Omega$, so $u \in \mathcal{A}_{3}$ and $u$ is a minimizer of $I$ in $\mathcal{A}_{3}$.
		
		The fact that any element of $\mathcal{A}_{i}$ is injective a.e.\ in $\Omega$ for each $i\in\{1,2,3\}$ is due to Theorem \ref{teorema 9.1 revisited}.
	\end{proof}
	
	Note that the particular case of $\mathcal{A}_{1}$ with $\Gamma=\partial\Omega$ does not need any assumptions in $U$ since, in this case, it is constant. Thanks to \cite[Proposition 2.5]{DaFoMaTr99} we can also assume $U$ to be tangentially quasiconvex instead of tangentially polyconvex.
	
	\section*{Acknowledgements}
	
	Both authors have been supported by the Spanish Agencia Estatal de Investigaci\'on through project  PID2021-124195NB-C32.
	C. Mora-Corral has also been supported by the Severo Ochoa Programme CEX2019-000904-S, the ERC Advanced Grant 834728 and by the Madrid Government (Comunidad de Madrid, Spain) under the multiannual Agreement with UAM in the line for the Excellence of the University Research Staff in the context of the V PRICIT (Regional Programme of Research and Technological Innovation).
	
	\bibliographystyle{abbrv}
	\bibliography{bibArt1}
\end{document}